\documentclass[notitlepage,11pt,reqno]{amsart}
\usepackage[foot]{amsaddr}
\usepackage[numbers,sort]{natbib}

\usepackage{amssymb,nicefrac,bm,upgreek,mathtools,verbatim}
\usepackage[final]{hyperref}
\usepackage[mathscr]{eucal}
\usepackage{dsfont}

\usepackage{color}
\definecolor{dmagenta}{rgb}{.4,.1,.5}
\definecolor{dblue}{rgb}{.0,.0,.5}
\definecolor{mblue}{rgb}{.0,.0,.8}
\definecolor{ddblue}{rgb}{.0,.0,.4}
\definecolor{dred}{rgb}{.6,.0,.0}
\definecolor{dgreen}{rgb}{.0,.5,.0}
\definecolor{Eeom}{rgb}{.0,.0,.5}

\usepackage[normalem]{ulem}

\usepackage[margin=1in]{geometry}
\allowdisplaybreaks

\newtheorem{lemma}{Lemma}[section]
\newtheorem{theorem}{Theorem}[section]
\newtheorem{proposition}{Proposition}[section]
\newtheorem{corollary}{Corollary}[section]

\theoremstyle{definition}
\newtheorem{definition}{Definition}[section]
\newtheorem{assumption}{Assumption}[section]

\newtheorem{example}{Example}[section]
\newtheorem{remark}{Remark}[section]

\numberwithin{equation}{section}

\hypersetup{
  colorlinks=true,
  citecolor=dblue,
  linkcolor=dblue,
  frenchlinks=false,
  pdfborder={0 0 0},
  naturalnames=false,
  hypertexnames=false,
  breaklinks}
\usepackage[capitalize,nameinlink]{cleveref}

\crefname{section}{Section}{Sections}
\crefname{subsection}{Subsection}{Subsections}
\crefname{condition}{Condition}{Conditions}
\crefname{hypothesis}{Hypothesis}{Conditions}
\crefname{assumption}{Assumption}{Assumptions}
\crefname{lemma}{Lemma}{Lemmas}
\crefname{claim}{Claim}{Claims}

\Crefname{figure}{Figure}{Figures}

\crefformat{equation}{\textup{#2(#1)#3}}
\crefrangeformat{equation}{\textup{#3(#1)#4--#5(#2)#6}}
\crefmultiformat{equation}{\textup{#2(#1)#3}}{ and \textup{#2(#1)#3}}
{, \textup{#2(#1)#3}}{, and \textup{#2(#1)#3}}
\crefrangemultiformat{equation}{\textup{#3(#1)#4--#5(#2)#6}}%
{ and \textup{#3(#1)#4--#5(#2)#6}}{, \textup{#3(#1)#4--#5(#2)#6}}%
{, and \textup{#3(#1)#4--#5(#2)#6}}

\Crefformat{equation}{#2Equation~\textup{(#1)}#3}
\Crefrangeformat{equation}{Equations~\textup{#3(#1)#4--#5(#2)#6}}
\Crefmultiformat{equation}{Equations~\textup{#2(#1)#3}}{ and \textup{#2(#1)#3}}
{, \textup{#2(#1)#3}}{, and \textup{#2(#1)#3}}
\Crefrangemultiformat{equation}{Equations~\textup{#3(#1)#4--#5(#2)#6}}%
{ and \textup{#3(#1)#4--#5(#2)#6}}{, \textup{#3(#1)#4--#5(#2)#6}}%
{, and \textup{#3(#1)#4--#5(#2)#6}}

\crefdefaultlabelformat{#2\textup{#1}#3}

\makeatletter
\DeclareRobustCommand\widecheck[1]{{\mathpalette\@widecheck{#1}}}
\def\@widecheck#1#2{%
    \setbox\z@\hbox{\m@th$#1#2$}%
    \setbox\tw@\hbox{\m@th$#1%
       \widehat{%
          \vrule\@width\z@\@height\ht\z@
          \vrule\@height\z@\@width\wd\z@}$}%
    \dp\tw@-\ht\z@
    \@tempdima\ht\z@ \advance\@tempdima2\ht\tw@ \divide\@tempdima\thr@@
    \setbox\tw@\hbox{%
       \raise\@tempdima\hbox{\scalebox{1}[-1]{\lower\@tempdima\box
\tw@}}}%
    {\ooalign{\box\tw@ \cr \box\z@}}}

\def\subsection{\@startsection{subsection}{0}%
\z@{\linespacing\@plus\linespacing}{\linespacing}%
{\bf}}

\makeatother

\DeclareMathOperator{\Prob}{\mathbb{P}} 
\newcommand{\D}{\mathrm{d}}          
\newcommand{\RR}{\mathbb{R}}         
\newcommand{\Rd}{{\mathbb{R}^d}}       
\newcommand{\NN}{\mathbb{N}}         
\newcommand{\Ind}{\mathds{1}}            

\newcommand{\grad}{\nabla}

\newcommand{\sB}{\mathscr{B}}    
\newcommand{\sU}{\mathscr{U}}
\newcommand{\cC}{\mathcal{C}}     
\newcommand{\cD}{\mathcal{D}}     
\newcommand{\cF}{\mathcal{F}}

\newcommand{\cB}{\mathcal{B}}

\newcommand{\cK}{\mathcal{K}}
\newcommand{\sM}{\mathscr{M}}     

\newcommand{\abs}[1]{\lvert#1\rvert}
\newcommand{\norm}[1]{\lVert#1\rVert}

\providecommand{\pro}[1]{(#1_t)_{t \geq 0}}

\providecommand{\semi}[1]{\{#1_t: t \geq 0\}}
\providecommand{\seq}[1]{(#1_n)_{n\in \mathbb{N}}}

\DeclareMathOperator{\Dom}{Dom}

\DeclareMathOperator{\Tr}{Tr}
\DeclareMathOperator{\diam}{diam}
\DeclareMathOperator{\inrad}{inrad}

\newcommand{\ex}{\mathbb{E}}

\newcommand{\Dst}{\mathrel{\stackrel{\makebox[0pt]{\mbox{\normalfont\tiny d}}}{=}}}

\DeclareMathOperator{\dist}{dist}

\DeclareMathOperator{\Psidel}{\Psi(-\Delta)}

\begin{document}

\title[Maximum principles]%
{\sc \textbf{Maximum principles and Aleksandrov-Bakelman-Pucci type estimates for non-local Schr\"odinger 
equations with exterior conditions}}

\author{Anup Biswas and J\'{o}zsef L\H{o}rinczi}

\address{Anup Biswas \\
Department of Mathematics, Indian Institute of Science Education and Research, Dr. Homi Bhabha Road,
Pune 411008, India, anup@iiserpune.ac.in}

\address{J\'ozsef L\H{o}rinczi \\
Department of Mathematical Sciences, Loughborough University, Loughborough LE11 3TU, United Kingdom,
J.Lorinczi@lboro.ac.uk}

\date{}


\begin{abstract}
We consider Dirichlet exterior value problems related to a class of non-local Schr\"odinger operators, whose
kinetic terms are given in terms of Bernstein functions of the Laplacian. We prove elliptic and parabolic
Aleksandrov-Bakelman-Pucci type estimates, and as an application obtain existence and uniqueness of weak solutions.
Next we prove a refined maximum principle in the sense of Berestycki-Nirenberg-Varadhan, and a converse. Also, we
prove a weak anti-maximum principle in the sense of Cl\'ement-Peletier, valid on compact subsets of the domain, and
a full anti-maximum principle by restricting to fractional Schr\"odinger operators. Furthermore, we show a maximum
principle for narrow domains, and a refined elliptic ABP-type estimate. Finally, we obtain Liouville-type theorems
for harmonic solutions and for a class of semi-linear equations. Our approach is probabilistic, making use of the
properties of subordinate Brownian motion.
\end{abstract}



\maketitle

\section{\textbf{Introduction}}
The techniques developed around the broad concept of extremal behaviour of (sub-/super-) solutions of boundary value
problems proved to be very successful in the analysis of partial differential equations. Currently there are various
refinements and generalizations of maximum principles in place, which have a deep impact on proofs of existence,
uniqueness, regularity, and various qualitative properties of solutions. Recently, new efforts have been made to
extend these techniques to integro-differential (i.e., non-local) equations as well. Our goal in this paper is to
further contribute to a developing of maximum principles for non-local equations with exterior conditions. We will
be concerned with three aspects of such problems, Aleksandrov-Bakelman-Pucci type estimates, refined maximum/anti-maximum
principles, and Liouville-type theorems.

Consider the elliptic operator $L = \sum_{i,j=1}^d a_{ij}(x)\partial_{x_i}\partial_{x_j}$, with a positive-definite
symmetric matrix $A = (a_{ij}(x))_{1\leq i,j\leq d}$, a given function $f$, a bounded domain $\cD \subset \Rd$, and
the boundary value problem
\begin{equation*}
\left\{\begin{array}{ll}
-Lu \leq f, \quad \text{in}\, \; \cD \\
\quad \;\, u \leq 0,\, \quad \text{on}\,\; \partial \cD.
\end{array} \right.
\end{equation*}
A fundamental result, which is now known as the Aleksandrov-Bakelman-Pucci (ABP) estimate, states that a solution
$u \in \cC^2(\cD) \cap \cC(\bar\cD)$ satisfies
$$
\sup_{\cD} u \leq c \left\|\frac{f}{(\det A)^{1/d}}\right\|_{L^d(\cD)},
$$
with a suitable constant $c = c(d,\cD) > 0$. Various generalizations of such estimates have been obtained in the last
decades. Extensions for uniformly elliptic fully nonlinear equations have been derived in \cite{C89,C95,CC95,ACP,I11}.
Less regular solutions, such as $u \in W^{2,p}_{\rm loc}(\cD) \cap \cC(\bar \cD)$ with a suitable $p < d$, have been
considered in \cite{FS84,E93,C95}, and for related Bony-estimates we refer to \cite{L83}. For extensions to cases of
viscosity solutions we refer to \cite{CCKS,DFQ,KS}, and for unbounded domains see \cite{C95,CDLV,BCDV,V03}. Apart
from elliptic equations, ABP estimates have been obtained also for parabolic equations \cite{K76,T85,CKS}. For a
recent survey presenting applications and further references see \cite{C17}. Although these estimates formulate
naturally in the terminology of analysis, it is interesting to note that a probabilistic counterpart has been obtained
in \cite{K72,K74}. Indeed, Krylov showed that if for a diffusion given by
$$
\D{X}_t = b_t \D{t} + \sigma_t \D{B}_t,
$$
where $\pro B$ is standard Brownian motion, and the drift and diffusion coefficients are chosen in such a way that
the right hand side exists as a stochastic integral, and furthermore,
$$
|b_t| \leq C_1 (\det\sigma_t^\top\sigma_t)^{1/d} \quad \mbox{and} \quad \Tr \sigma_t^\top\sigma_t \leq C_2, \qquad t > 0,
\;\; \Prob_W-\mbox{a.s.},
$$
hold with some constants $C_1, C_2  >0$, then the following expectation with respect to Wiener measure $\Prob_W$ satisfies
$$
\ex^x_{\Prob_W} \left[\int_0^{\uptau_\cD} (\det\sigma_t^\top\sigma_t)^{1/d}|f(X_t)| \D{t}\right] \leq C\|f\|_{L^d(\cD)},
$$
for every $f \in L^d(\cD)$, where $\cD \subset \Rd$ is any bounded domain, $\uptau_\cD = \inf\{t > 0: \, X_t \not\in \cD\}$
is the first exit time of the process from $\cD$, and $C = C(d,C_1,\diam \cD)$ is a suitable constant.

An ABP-type estimate is a specific expression of more general maximum principles used in PDE theory and harmonic analysis.
Consider, more generally than above, the operator
$$
L = \sum_{i,j=1}^d a_{ij}(x)\partial_{x_i}\partial_{x_j} + \sum_{i=1}^d b_i(x)\partial_{x_i} + c(x)
$$
with uniform ellipticity condition and sufficient regularity of the coefficients ($a_{ij}$ are continuous, and $\abs{b},
\abs{c}$ are bounded). Recall that a classical version of the maximum principle for a bounded domain $\cD \subset \Rd$ says
that if
$$
Lu \geq 0 \; \mbox{in $\cD$} \quad \mbox{with} \quad c\leq 0\; \mbox{in $\cD$}, \quad \mbox{and}\quad \limsup_{x\to\partial\cD} u(x) \leq 0,
$$
then $u \leq 0$ in $\cD$. There are well-known conditions ensuring that the maximum principle holds for cases when $u \in
\cC^2$ or $u \in W^{2,d}_{\rm loc}$. Also, under conditions of sufficient regularity of $L$ and $\partial \cD$ it has been
established that the maximum principle holds exactly when the principal Dirichlet eigenvalue of $-L$ in $\cD$ is positive.
Using this relationship, Berestycki, Nirenberg and Varadhan defined a  generalized principal eigenvalue given by
$$
\lambda(L,\cD) = \sup \{\lambda: \, \mbox{there exists $w > 0$ in $\cD$ such that $Lw + \lambda w \leq 0$} \}
$$
and proposed in \cite{BNV} what is now called a refined maximum principle, valid for any bounded domain. The key step in
their construction for irregular boundaries was to prescribe weaker conditions, namely only for a sequence of points in
$\cD$ approaching the boundary for which a solution of the equation $(L-c)v = -1$ vanishes on $\partial\cD$. Denoting such
a sequence by $\seq x$ and $x_n \rightsquigarrow \partial \cD$ to mean that $\lim_{n\to\infty}v(x_n) = 0$, the refined
maximum principle says that if
$$
Lu \geq 0 \; \mbox{in $\cD$}, \;\; \mbox{$u$ is bounded from above}, \;\; \mbox{and} \;\; \limsup_{x_n\rightsquigarrow
\partial\cD} u(x_n) \leq 0,
$$
then $u \leq 0$ in $\cD$. Furthermore, the authors proved that the refined maximum principle holds for $L$ exactly when the
generalized principal eigenvalue $\lambda(L,\cD) > 0$. It is worthwhile to note that this construction has an intrinsic
probabilistic meaning. Indeed, $v(x)$ corresponds to the mean of the first exit time of the diffusion generated by $L-c$
starting from a point $x \in \cD$, and thus boundary conditions are set only on those points of $\partial \cD$ which can be
reached by an exit event through a sequence $\seq x$. For subsequent developments on the characterization of the generalized
principal eigenvalue and further generalizations we refer to \cite{NP,PT,BR,BCPR}, and for a book-length discussion of the
probabilistic aspects to \cite{P08}.

Another type of results are the anti-maximum principles related to a sign-reverting phenomenon, first observed by Cl\'ement
and Peletier \cite{CP}. An initial version of this has been established for the boundary value problem
$$
-\Delta u = \lambda u + f \; \mbox{in $\cD$} \quad \mbox{with} \quad u = 0 \; \mbox{on $\partial\cD$},
$$
where the boundary $\partial\cD$ is assumed to be smooth and $f \in L^p(\Rd)$, $p>1$. Choosing $f \geq 0$, strictly positive
on a non-zero measure subset of $\cD$, the maximum principle implies that $u > 0$ if $\lambda < \lambda_1$, where $\lambda_1$
is the principal Dirichlet eigenvalue of the Laplacian. However, the authors proved that positivity does not hold for
arbitrarily large $\lambda$ beyond the principal Dirichlet eigenvalue, and for $p > d$ there exists $\lambda(f) > \lambda_1$
such that $u < 0$ for all $\lambda \in (\lambda_1, \lambda(f))$. Subsequently, it has been shown that $p > d$ is a sharp
condition, and further related results have been obtained for more general cases, including classical Schr\"odinger operators;
we refer the reader to \cite{H81,B95,T96,S97,P99,AFT,CS01,GGP02}.

Although recently much research has been done on general non-local equations, results on maximum principles are scarce;
see, for instance, the open problems section in \cite{Mwiki}.
The first version of ABP was obtained in \cite{CS-09} for nonlinear stable-like operators, where the estimate only used
$L^\infty$ norm of $f$. Recently, a more quantitative version involving a combination of $L^d$ and $L^\infty$ norms of $f$
is proved in \cite{GS12} for a class of fractional operators comparable, in some sense, with the fractional Laplacian.
The authors replace the usual concept of convex envelope by another object and rely essentially on a use of the Riesz
potential to obtain their estimates, and also discuss the difficulties for which more general non-local operators cannot
be covered in their framework. Another feature is that \cite{CS-09, GS12} do not consider the zeroth order term in their
equations. The paper \cite{MS17} uses more general operators with a non-degenerate second order term and establishes ABP
estimates for a class of uniformly elliptic and parabolic non-local equations. Here one can see the non-local part as a
perturbation of the usual second order elliptic operator and thus it becomes challenging to obtain a similar estimate for
a purely non-local operator. Some further related works include \cite{GL,CLD}.

A third direction of development we consider are Liouville-type theorems. Recall that the classical Liouville theorem for
harmonic functions says that any non-negative solution of $\Delta u = 0$ in $\Rd$ is a constant. Liouville-type results
for the Laplacian have been extended by Gidas and Spruck \cite{GS} to non-negative solutions of the semi-linear elliptic
equation
$$
\Delta u + c\,u^p = 0 \;\; \mbox{in $\Rd$},
$$
showing that if $c > 0$ and $p \in (1, \frac{d+2}{d-2})$, then $u \equiv 0$; for further developments we refer to
\cite{BCN,CDC} and references therein. For non-local equations, \cite{MC06, FQ11} considered Liouville theorems for the
above problem with operators comparable to the fractional Laplacian. In \cite{QX16} non-existence of positive viscosity
solutions was similarly addressed for Lane-Emden systems involving fractional Laplacians. In \cite{FW16} a larger class
of non-local operators is considered for harmonic functions, however, with a polynomial decay of its jump measure at
infinity.

In the present paper we derive and prove results in the above three directions for non-local Schr\"odinger equations.
These problems involve non-local Schr\"odinger operators of the form
$$
H^{\cD, V} = \Psi(-\Delta) + V,
$$
restricted to suitable function spaces over bounded domains $\cD \subset \Rd$ and possibly with $V \equiv 0$. The kinetic
term in $H^{\cD, V}$ is given by a Bernstein function $\Psi$ of the Laplacian, and the potential term is given by a
multiplication operator $V$ (for details see Section 2 below). Such operators have been considered in
\cite{HIL12,HL12,KL15,KL17} in detail, and they have a number of applications in relativistic quantum theory, anomalous
transport and other fields. Another important aspect is that the operators $-\Psi(-\Delta)$ are infinitesimal generators
of a class of L\'evy processes, and therefore, such non-local equations have close ties with problems in probability and
stochastic control. An example is the fractional Laplacian $(-\Delta)^{\alpha/2}$, which is currently much investigated
in both analysis and probability. However, other choices of $\Psi$ relate with many other applications (for a catalogue of
Bernstein functions with detailed descriptions see \cite{SSV}), and we should emphasize that an operator $\Psi(-\Delta)$
different from the fractional Laplacian may involve in general very different properties. Also, a rapidly growing
literature on non-local equations reveals that such equations display a number of new properties and behaviours, which
differ substantially from their PDE analogues based on the classical (local) Laplacian.

In contrast with the existing literature, our approach to obtaining maximum principles is probabilistic, using a functional
integral representation of the solution semigroup related to the non-local equations we consider. This has the benefit
of being rather robust in tackling the difficulties arising from the non-locality of the operators close to the boundary
of the domain, allowing to obtain such results for a large class of non-local operators. As it will be seen below
(Proposition \ref{P1.2}), for the Dirichlet exterior condition problem
$$
H^{\cD, V}\varphi = f \;\; \mbox{in $\cD$} \quad \mbox{with} \quad \varphi = 0 \;\; \mbox{in $\Rd\setminus\cD$},
$$
we have the functional integral representation
$$
\varphi(x) = \ex^x\left[e^{-\int_0^{t\wedge\uptau_\cD} V(X_s)\,\D{s}} \varphi(X_{t\wedge\uptau_\cD})\right] +
\ex^x\left[\int_0^{t\wedge\uptau_\cD} e^{-\int_0^s V(X_{r}) \, \D{r}} f(X_s)\, \D{s}\right],
$$
where $\pro X$ is the jump L\'evy process generated by $-\Psi(-\Delta)$, $\uptau_\cD$ is its first exit time from $\cD$,
and the expectation is taken with respect to the probability measure of the process starting at $x \in \Rd$. Given the
specific form of the kinetic part of $H^{\cD, V}$, the jump process can be described in some detail (in fact, it is a
subordinate Brownian motion $X_t = B_{S_t^\Psi}$, i.e., Brownian motion sampled at random times given by a subordinator
$\pro {S^\Psi}$ uniquely determined by $\Psi$), which is sufficient for us to be able to control such expectations. This
will be explained in more detail in Section 2 below. To the best of our knowledge, developing such tools to prove maximum
principles for non-local equations has not been attempted in the literature before. We also note that although we focus
mainly on operators related to subordinate Brownian motion, going well beyond the fractional Laplacian and related
stable processes, our approach is more accommodating and works also for a larger class of operators related to more
general Markov processes (see Remark~\ref{R3.2} below). In particular, we are not aware of similar results covering, for
instance, operators such as $-\Delta + b(-\Delta)^{\alpha/2}$ or $(-\Delta + m^{2/\alpha})^{\alpha/2} - m$, $\alpha \in
(0,2)$, $b, m > 0$. Moreover, we emphasize that our ABP-type estimates do not involve a contact set like used in
\cite{CS-09, GS12}, while we use $L^p$ norm of the source term $f$.

Our main results are as follows. In Section 3 first we obtain elliptic (Theorem \ref{T1.6}) and parabolic (Theorem
\ref{parabo}) ABP-type estimates for a large class of equations related to $H^{\cD, V}$. As an application, in Theorem
\ref{T3.3} we prove existence and uniqueness of solutions for a Dirichlet exterior value problem for $H^{\cD, V}$. Next,
in Section 4, under a mild probabilistic condition on boundary regularity we derive and prove a stochastic representation
of the principal eigenfunction of the non-local Schr\"odinger operator (Theorem \ref{T1.1}). This will then allow us to
obtain a number of maximum principles for equations related to $H^{\cD, V}$. These maximum principles appear to be new
in the context of non-local operators. Theorem \ref{T1.2} gives a refined maximum principle, and Theorem \ref{conv} shows
a converse. Theorem \ref{T1.4} presents a weak anti-maximum principle, which we call `weak' due to the fact that it holds
for compact subsets of the domain. Since Hopf's lemma is available for fractional Laplacians, we can prove a full
anti-maximum principle for fractional Schr\"odinger operators in Theorem \ref{antimaxx}. Due to the special role of narrow
domains in the sufficiency of classical maximum principles, we show a maximum principle for $H^{\cD, V}$ for such domains
in Theorem \ref{narrow}. Making use of the stochastic representation in Theorem \ref{T1.1}, we also obtain a refined
elliptic ABP-type estimate (Theorem \ref{T4.7}). Finally, we present Liouville-type theorems in a novel approach, first
for harmonic functions with respect to $\Psi(-\Delta)$ in Theorem \ref{T4.8}, and next for a class of semi-linear non-local
equations in Theorem \ref{T4.9} by using recurrence properties of the related random process.

\section{\textbf{Non-local Schr\"odinger operators}}
In this section we briefly describe the operators involved in the non-local equations studied
in Sections 3-4. Recall that a Bernstein function is a non-negative completely monotone function, i.e., an element
of
$$
\mathcal B = \left\{f \in \cC^\infty((0,\infty)): \, f \geq 0 \;\; \mbox{and} \:\; (-1)^n\frac{d^n f}{dx^n} \leq 0,
\; \mbox{for all $n \in \mathbb N$}\right\}.
$$
In particular, Bernstein functions are increasing and concave. Below we will restrict to the subset
$$
{\mathcal B}_0 = \left\{f \in \mathcal B: \, \lim_{u\downarrow 0} f(u) = 0 \right\}.
$$
Let $\mathcal M$ be the set of Borel measures $\nu$ on $\RR \setminus \{0\}$ with the property that
$$
\nu((-\infty,0)) = 0 \quad \mbox{and} \quad \int_{\RR\setminus\{0\}} (y \wedge 1) \nu(dy) < \infty.
$$
Bernstein functions $\Psi \in {\mathcal B}_0$ can be represented in the form
$$
\Psi(u) = bu + \int_{(0,\infty)} (1 - e^{-yu}) \nu(\D{y})
$$
with $b \geq 0$, and the map $[0,\infty) \times \mathcal M \ni (b,\nu) \mapsto \Psi \in {\mathcal B}_0$ is
bijective.

\medskip
Below we will often use a class of Bernstein functions singled out by the following property.
\begin{assumption}
\label{WLSC}
The function $\Psi$ is said to satisfy a
\begin{enumerate}
\item[(i)]
weak lower scaling (WLSC) property with parameters ${\underline\mu} > 0$, $\underline{c} \in(0, 1]$ and
$\underline{\theta}\geq 0$, if
$$
\Psi(\gamma u) \;\geq\; \underline{c}\, \gamma^{\underline\mu} \Psi(u), \quad u>\underline{\theta}, \; \gamma\geq 1.
$$
\item[(ii)]
weak upper scaling (WUSC) property with parameters $\bar\mu > 0$,
$\bar{c} \in[1, \infty)$ and $\bar{\theta}\geq 0$, if
$$
\Psi(\gamma u) \;\leq\; \bar{c}\, \gamma^{\bar{\mu}} \Psi(u), \quad u>\bar{\theta}, \; \gamma\geq 1.
$$
\end{enumerate}
\end{assumption}

\begin{example}\label{Eg2.1}
Some important examples of $\Psi$ satisfying Assumptions \ref{WLSC}(i) and \ref{WLSC}(ii) include the following
cases with the given parameters, respectively:
\begin{itemize}
\item[(i)]
$\Psi(u)=u^{\alpha/2}, \, \alpha\in(0, 2]$, with ${\underline\mu} = \frac{\alpha}{2}$, $\underline{\theta}=0$,
and $\bar\mu = \frac{\alpha}{2}$, $\bar{\theta}=0$.
\item[(ii)]
$\Psi(u)=(u+m^{2/\alpha})^{\alpha/2}-m$, $m> 0$, $\alpha\in (0, 2)$, with ${\underline\mu} = \frac{\alpha}{2}$,
$\underline{\theta}=0$ and $\bar\mu = 1$, $\bar{\theta}=0$.
\item[(iii)]
$\Psi(u)=u^{\alpha/2} + u^{\beta/2}, \, \alpha, \beta \in(0, 2]$, with ${\underline\mu} = \frac{\alpha}{2}
\wedge \frac{\beta}{2}$, $\underline{\theta}=0$ and $\bar\mu = \frac{\alpha}{2} \vee \frac{\beta}{2}$,
$\bar{\theta}=0$.
\item[(iv)]
$\Psi(u)=u^{\alpha/2}(\log(1+u))^{-\beta/2}$, $\alpha \in (0,2]$, $\beta \in [0,\alpha)$ with
${\underline\mu}=\frac{\alpha-\beta}{2}$, $\underline{\theta}=0$ and $\bar\mu=\frac{\alpha}{2}$, $\bar{\theta}=0$.
\item[(v)]
$\Psi(u)=u^{\alpha/2}(\log(1+u))^{\beta/2}$, $\alpha \in (0,2)$, $\beta \in (0, 2-\alpha)$, with
${\underline\mu}=\frac{\alpha}{2}$, $\underline{\theta}=0$ and $\bar\mu=\frac{\alpha+\beta}{2}$, $\bar{\theta}=0$.
\end{itemize}
\end{example}
\begin{remark}
It is known \cite[Lem.~11]{BGR14b} that $\Psi$ has the WLSC property with parameters ${\underline\mu}$,
$\underline{c}$ and $\underline\theta$ if and only if $\Psi(u) u^{-{\underline\mu}}$ is comparable to a
non-decreasing function in $(\underline\theta, \infty)$. Also, $\Psi$ has the WUSC property with
parameters ${\bar\mu}$, $\bar{c}$ and $\bar\theta$ if and only if $\Psi(u) u^{-{\bar\mu}}$ is comparable
to a non-increasing function in $(\bar\theta, \infty)$. These scaling properties are also related to the
Matuszewska indices, for a discussion see \cite[Rem.~2]{BGR14b}.
\end{remark}

\medskip
Next consider the operator
\begin{equation}
\label{BernLapl}
H = H_0 +  V := \Psidel + V,
\end{equation}
which we call a non-local Schr\"odinger operator with kinetic term $H_0 = \Psidel$ and potential $V$, where $\Psi
\in \cB_0$. The operator $H_0$ can be defined through functional calculus by using the spectral decomposition of
the Laplacian. It is a pseudo-differential operator with Fourier multiplier
$$
\widehat{H_0 f}(y) = \Psi(|y|^2)\widehat f(y), \quad y \in \Rd, \; f \in  \Dom(H_0),
$$
and domain $\Dom(H_0)=\big\{f \in L^2(\Rd): \Psi(|\cdot|^2) \widehat f \in L^2(\Rd) \big\}$. By general arguments
it can be seen that $H_0$ is a positive, self-adjoint operator with core $C_{\rm c}^\infty(\Rd)$.

For simplicity, we choose the potential $V \in L^\infty(\Rd)$, so the non-local Schr\"odinger operator $H$ can be
defined as a self-adjoint operator in terms of perturbation theory. However, we note that this restriction is not
necessary, and we could use Kato-class potentials also allowing local singularities. For more details we refer to
\cite{BL} and references therein.

In what follows, we will use a stochastic representation of the semigroup $\{e^{-tH}: \, t \geq 0\}$. This is
obtained by using the fact that Bernstein functions are related to subordinators. Recall that a one-dimensional L\'evy
process $\pro S$ on a probability space $(\Omega_S, {\mathcal F}_S, \mathbb P_S)$ is called a subordinator whenever
it satisfies $S_s \leq S_t$ for $s \leq t$, $\mathbb P_S$-almost surely. A basic fact is that the Laplace transform
of a subordinator is given by a Bernstein function, i.e.,
\begin{equation*}
\ex_{\mathbb P_S} [e^{-uS_t}] = e^{-t\Psi(u)}, \quad t \geq 0,
\end{equation*}
holds, where $\Psi \in {\mathcal B}_0$. Moreover, there is a bijection between the set of subordinators on a given
probability space and Bernstein functions in $\cB_0$. In our notation below, we will write $\pro {S^\Psi}$ for the
unique subordinator associated with Bernstein function $\Psi$. In Example~\ref{Eg2.1} above (i) corresponds to an
$\alpha/2$-stable subordinator, (ii) to a relativistic $\alpha/2$-stable subordinator, (iii) to sums of independent
subordinators of different indices, etc. For a detailed discussion of Bernstein functions and subordinators we
refer to \cite{SSV}.

Let $\pro B$ be $\Rd$-valued Brownian motion on Wiener space $(\Omega_W,{\mathcal F}_W, \mathbb P_W)$, with variance
$\ex_{\Prob_W} [B_t^2] = 2t$, $t\geq 0$. Also, let $\pro {S^\Psi}$ be an independent subordinator. The random process
$$
\Omega_W \times \Omega_S \ni (\omega,\varpi) \mapsto B_{S_t(\varpi)}(\omega) \in \Rd
$$
is called subordinate Brownian motion under $\pro {S^\Psi}$. Every subordinate Brownian motion is a L\'evy process,
satisfying the strong Markov property, and apart from the trivial case generated by $\Psi(u) = u$ they have paths
with jump discontinuities. For simplicity, we will denote a subordinate Brownian motion by $\pro X$, its probability
measure for the process starting at $x \in \Rd$ by $\mathbb P^x$, and expectation with respect to this measure by
$\ex^x$.

The relationship between the operator $H$ given by \eqref{BernLapl} and these processes is expressed by a Feynman-Kac
type formula obtained in \cite{HIL12}. This relies on the fact that the infinitesimal generator of $\pro X$ obtained by
subordinating Brownian motion with a subordinator of Laplace exponent $\Psi$, is the operator $-H_0 = -\Psidel$. Under
perturbation by $V$ we then have the formula
$$
e^{-tH} f(x) = \ex^x [e^{-\int_0^t V(X_s) \D s} f(X_t)], \quad t \geq 0, \, x \in \Rd, \, f \in L^2(\Rd).
$$
Also, subordination gives the expression
\begin{equation*}
\Prob (X_t \in E) = \int_0^\infty \Prob_W(B_s \in E)\Prob_S(S_t \in \D s),
\end{equation*}
for every measurable set $E$. For further details on non-local Schr\"odinger operators and related jump processes we
refer to \cite{HIL12,HL12,KL15,KL17} and references therein.

It is straightforward to see that $\Psi \in \cB_0$ satisfying Assumption~\ref{WLSC}(i) also satisfies the
Hartman-Wintner condition
\begin{equation}
\label{HW}
\lim_{\abs{u}\to\infty} \, \frac{\Psi(u^2)}{\log\abs{u}}=\infty.
\end{equation}
It is known that under this condition the subordinate Brownian motion $\pro X$ has a bounded continuous transition
probability density $q_t(x,y)$, see \cite{KS13}. It follows also that $q_t(x, y)=q_t(x-y)$ and $q_t(\cdot)$ is radially
decreasing.

We close this section by presenting the following estimate, which will be useful below.
\begin{lemma}\label{L2.1}
Let Assumption~\ref{WLSC}(i) hold and $\Psi \in \cB_0$ be strictly increasing. Then there exist positive constants $\kappa_1,
\kappa_2$ such that
\begin{equation}\label{EL2.1A}
q_t(x)\leq \kappa_1 t^{-\frac{d}{2{\underline\mu}}}, \quad  x \in \Rd, \; t\in (0, \kappa_2].
\end{equation}
\end{lemma}
\begin{proof}
Let $\Phi(u)=\Psi(u^2)$. By our assumptions on $\Psi$ it follows that $\Phi$ strictly increasing and
\begin{equation}\label{EL2.1B}
\Phi(\gamma u)\geq \underline{c}\,\gamma^{2{\underline\mu}}\, \Phi(u), \quad \gamma\geq 1, \; u\geq \sqrt{\underline \theta}.
\end{equation}
Thus by \cite[Prop.~19]{BGR14b} there exists a constant $C = C(d, {\underline\mu})$ such that
\begin{equation}\label{EL2.1C}
q_t(x)\leq C \left(\Phi^{-1}\left(\frac{1}{t}\right)\right)^d, \;\; t>0, \quad \text{and}\quad
t\Phi(\sqrt{\underline\theta})<\frac{1}{\pi^2}.
\end{equation}
We may assume $\underline\theta>0$ with no loss of generality. Thus from \eqref{EL2.1B} it is seen that
$$
\Phi(\gamma^{\frac{1}{2{\underline\mu}}} \sqrt{\underline\theta})
\geq \underline{c}\, \gamma\, \Phi(\sqrt{\underline\theta}),\quad  \gamma\geq 1,
$$
implying
$$
\Phi^{-1}(\underline{c}\, \gamma\, \Phi(\sqrt{\underline\theta}))\leq \gamma^{\frac{1}{2{\underline\mu}}} \sqrt{\underline\theta},
\quad \gamma\geq 1.
$$
Write $\kappa_3=\underline{c}\Phi(\sqrt{\underline\theta})$. Then for every $\gamma\geq \kappa_3$ we obtain
$$
\Phi^{-1}(\gamma)=\Phi^{-1}\left(\frac{\gamma}{\kappa_3}\kappa_3\right)
\leq \sqrt{\underline\theta}\, \kappa_3^{-\frac{1}{2{\underline\mu}}}\gamma^{\frac{1}{2{\underline\mu}}}.
$$
Hence \eqref{EL2.1A} follows from \eqref{EL2.1C} by choosing
$$
\kappa_2 = \frac{1}{\kappa_3} \wedge \frac{1}{2\pi^2 \Phi(\sqrt{\underline\theta})} \quad \text{and}\quad
\kappa_1= C\, \left(\sqrt{\underline\theta}\, \kappa_3^{-\frac{1}{2{\underline\mu}}}\right)^d\,.
$$
\end{proof}

\medskip
\section{\textbf{Aleksandrov-Bakelman-Pucci estimates}}

\subsection{Elliptic and parabolic ABP-type estimates}
In this section we derive Aleksandrov-Bakelman-Pucci (ABP) estimates of elliptic and parabolic types using a
probabilistic approach. First we consider the elliptic case.

We use the notation
$$
\uptau_\cD=\inf\{t>0\; :\; X_t\notin \cD\}
$$
for the first exit time of $\pro X$ from a domain $\cD$. A standing assumption in this paper is the following.
\begin{assumption}\label{As3.1}
$\cD \subset \Rd$ is a bounded domain and all points of $\partial\cD$ are regular, i.e., for every
$z\in\partial\cD$ we have $\Prob^z(\uptau_\cD=0)=1$.
\end{assumption}
\noindent
It  follows from \cite[Lem.~2.9]{BGR15} that every domain $\cD$ with the
exterior cone condition has a regular boundary in the above sense, provided $\Psi$ is unbounded.

We denote the diameter of the domain $\cD$ by $\diam \cD$. In case $\diam \cD < \infty$, it is known that
$\sup_{x\in\cD}\ex^x[\uptau_\cD]<\infty$. Below we will need the following lemma.
\begin{lemma}\label{L3.1}
Let $\cD \subset \Rd$ be a domain such that $\diam \cD < \infty$. For every $k\in\NN$ we have
\begin{equation}\label{EL3.1A}
\sup_{x\in\cD}\ex^x[\uptau_\cD^k]\;\leq\; k! \left(\sup_{x\in\cD}\ex^x[\uptau_\cD]\right)^k.
\end{equation}
Moreover, there exists a constant $c_k = c_k(d)$ such that
$$
\sup_{x\in\cD}\ex^x[\uptau_\cD^k]\;\leq\;  \frac{c_k}{(\Psi([\diam \cD]^{-2}))^k}.
$$
\end{lemma}

\begin{proof}
For simplicity, denote in this proof $\uptau=\uptau_\cD$. Recall that $\pro X$ is a strong Markov process
with respect to its natural filtration $\pro\cF$. Using the strong Markov property, for every $x\in\cD$ and $k\geq 2$
we have
\begin{align*}
\ex^x[\uptau^k] &=\ex^x\left[\int_0^\infty k (\uptau-t)^{k-1} \Ind_{\{\uptau> t\}} \, \D{t}\right]
\\
&=\ex^x\left[\int_0^\infty k \ex^x\left [(\uptau-t)^{k-1} \Ind_{\{\uptau> t\}}\Big|
\mathcal{F}_{\uptau\wedge t}\right] \, \D{t}\right]
\\
&=\ex^x\left[\int_0^\infty k  \Ind_{\{\uptau> t\}} \ex^{X_t}[\uptau^{k-1}] \, \D{t}\right]
\\
&\leq k \sup_{x\in\cD} \ex^{x}[\uptau^{k-1}]\, \sup_{x\in\cD} \ex^{x}\left [\uptau\right].
\end{align*}
This implies
$$
\sup_{x\in\cD} \ex^x[\uptau^k]\leq k \sup_{x\in\cD} \ex^{x}[\uptau^{k-1}]\,
\sup_{x\in\cD} \ex^{x}[\uptau]\leq \cdots\leq k!\, \left(\sup_{x\in\cD}
\ex^{x}[\uptau]\right)^k.
$$
This proves the first part of the lemma. The remaining part of the claim follows from \eqref{L3.1} and
\cite[Rem.~4.8]{SL}.
\end{proof}

The following will be key objects in this paper. Consider the operator $H$ as given by \eqref{BernLapl}. When
applied in an equation for a domain $\cD$, we use the notation $H^{\cD, V}$. Also, note that in the second
definition below (parabolic case) we assume, more generally, that $V: \RR^+ \times \Rd \to \RR$.
\begin{definition}\label{D3.1}
A function $\varphi\in\cC(\Rd)$ is said to be an \emph{(elliptic) weak sub-solution} to
$$
H^{\cD, V} \varphi \leq f \quad \mbox{in}\,\; \cD,
$$
whenever for every $t>0$ and $x\in\cD$,
\begin{equation}\label{D3.1A}
\varphi(x) \leq \ex^x\left[e^{-\int_0^{t\wedge\uptau_\cD} V(X_s)\,\D{s}} \varphi(X_{t\wedge\uptau_\cD})\right] +
\ex^x\left[\int_0^{t\wedge\uptau_\cD} e^{-\int_0^s V(X_{r}) \, \D{r}} f(X_s)\, \D{s}\right]
\end{equation}
holds. A function $\varphi\in \cC([0, T]\times\Rd)$ is said to be a \emph{(parabolic) weak sub-solution} of
$$
-\partial_t\varphi + H^{\cD, V}\varphi\leq f \quad \mbox{in}\,\; Q_T=[0, T)\times\cD,
$$
whenever for every $t\in[0, T)$ and $x\in\cD,$ we have
\begin{align}\label{ED3.1B}
\varphi(t, x)
&\leq \ex^x\left[e^{-\int_0^{(T-t)\wedge\uptau_\cD} V(t+s, X_s)\,\D{s}}
\varphi(t+(T-t)\wedge\uptau_\cD, X_{(T-t)\wedge\uptau_\cD})\right]
\nonumber
\\
&\,\quad +
\ex^x\left[\int_0^{(T-t)\wedge\uptau_\cD} e^{-\int_0^s V(t+r, X_{r}) \, \D{r}} f(t+s,X_s)\, \D{s}\right].
\end{align}
\end{definition}

Now we are ready to prove our elliptic ABP-type estimate.
\begin{theorem}[\textbf{Elliptic ABP estimate}]
\label{T1.6}
Let $\Psi \in \cB_0$ be strictly increasing and satisfy Assumption~\ref{WLSC}(i). Furthermore, let $V\geq 0$, and
$\varphi$ be any bounded weak sub-solution of
$$
H^{\cD, V}\varphi \leq f \quad \mbox{in}\,\; \cD,
$$
with $f\in L^p(\cD)$, for some $p>\frac{d}{2{\underline\mu}}$. Then there exists a constant $C = C(p, d, \Psi)$ such that
\begin{equation}\label{ET1.6AA}
\sup_{\cD} \varphi^+ \leq \sup_{\cD^c} \varphi^+ +
C \left(1 + \frac{1}{\left(\Psi([\diam \cD]^{-2})\right)^\frac{k}{p'}}\right) \norm{f}_{p, \cD},
\end{equation}
where $k=\lceil \frac{p}{p-1}\rceil + 1$, $\frac{1}{p} + \frac{1}{p'}=1$, and $\norm{\cdot}_{p, \cD}$ denotes the $L^p$
norm on $\cD$.
\end{theorem}

\begin{proof}
For simplicity of notation we extend $f$ by zero outside of $\cD$. It is readily seen from \eqref{D3.1A} that
\begin{align*}
\varphi(x) &\leq \ex^x\left[e^{-\int_0^{t\wedge\uptau_\cD} V(X_s)\,\D{s}} \varphi^+(X_{t\wedge\uptau_\cD})\right] +
\ex^x\left[\int_0^{t\wedge\uptau_\cD} e^{-\int_0^s V(X_{r}) \, \D{r}} f(X_s)\, \D{s}\right]
\\
&\leq \ex^x\left[\varphi^+(t\wedge\uptau_\cD)\right] +
\ex^x\left[\int_0^{t\wedge\uptau_\cD} \abs{f(X_s)}\, \D{s}\right],
\end{align*}
where $\varphi^+$ denotes the positive part of $\varphi$. Letting $t\to\infty$ and applying standard convergence
theorems, we get
\begin{equation}\label{ET1.6B}
\varphi(x)\leq \sup_{\cD^c}\varphi^+ + \ex^x\left[\int_0^{\uptau_\cD} \abs{f(X_s)}\, \D{s}\right].
\end{equation}
Thus to obtain \eqref{ET1.6AA} we only need to estimate the rightmost term in \eqref{ET1.6B}. Note that
by Lemma~\ref{L2.1} we have
$$
q_s(x, y)\leq \kappa_1 s^{-\frac{d}{2{\underline\mu}}}, \quad s\in (0, \kappa_2], \; x, y \in \Rd,
$$
for some constants $\kappa_1, \kappa_2$, where $q_s$ is the transition probability density of $\pro X$ at time $s$.
Since for $s\geq \kappa_2$ we have
$$
\sup_{x \in \Rd} q_s(x) =\frac{1}{(2\pi)^d}\int_{\Rd} e^{-i x\cdot y} e^{-s\Psi(\abs{y}^2)} \D y
\leq
\frac{1}{(2\pi)^d}\int_{\Rd}  e^{-s\Psi(\abs{y}^2)} \D y
\leq
\frac{1}{(2\pi)^d}\int_{\Rd}  e^{-\kappa_2\Psi(\abs{y}^2)} \D y,
$$
we obtain
\begin{equation}\label{AB100}
\sup_{s\geq \kappa_2}\sup_{x, y \in \Rd}q_s(x, y)\leq q_{\kappa_2}(0).
\end{equation}
Next note that
\begin{align*}
\ex^x\left[\int_0^{\uptau_\cD} \abs{f(X_s)}\, \D{s}\right]
&=
\ex^x\left[\int_0^{\infty}\Ind_{\{\uptau_\cD>s\}} \abs{f(X_s)}\, \D{s}\right]
\\
& \leq
\ex^x\left[\int_0^{\kappa_2} \abs{f(X_s)}\, \D{s}\right] +
\ex^x\left[\int_{\kappa_2}^{\infty}\Ind_{\{\uptau_\cD>s\}} \abs{f(X_s)}\, \D{s}\right].
\end{align*}
The first term above can be estimated as
\begin{align*}
\ex^x\left[\int_0^{\kappa_2} \abs{f(X_s)}\, \D{s}\right] &=\int_0^{\kappa_2} \int_{\Rd} \abs{f(y)} q_s(x, y)\D{y} \D{s}
\leq \norm{f}_{p} \int_0^{\kappa_2} \norm{q_s(x, \cdot)}_{p'} \D{s}
\\
&= \norm{f}_{p} \int_0^{\kappa_2} \left[\int_{\Rd} (q_s(x, y))^{p'} \D{y}\right]^{\nicefrac{1}{p'}} \D{s}
\\
&\leq \norm{f}_{p} \int_0^{\kappa_2} \left[(\kappa_1 s^{-\frac{d}{2{\underline\mu}}})^{p'-1} \int_{\Rd} q_s(x, y) \D{y}\right]^{\nicefrac{1}{p'}} \D{s}
\\
&= \kappa_1^{\nicefrac{1}{p}} \norm{f}_{p} \int_0^{\kappa_2} s^{-\frac{d}{2{\underline\mu} p}} \D{s}
= \frac{2{\underline\mu} p}{2{\underline\mu} p - d}\, \kappa_1^{\nicefrac{1}{p}}\, \kappa_2^{\frac{2{\underline\mu} p -d}{2{\underline\mu} p}} \norm{f}_{p} \,,
\end{align*}
where $p, p'$ are H\"older-conjugate exponents and $\norm{\cdot}_p$ denotes the $L^p$ norm over $\Rd$.
To deal with the second term choose $k\in\NN$ with $k>p'$. Then
\begin{align*}
\ex^x\left[\int_{\kappa_2}^{\infty}\Ind_{\{\uptau_\cD>s\}} \abs{f(X_s)}\, \D{s}\right]
& \leq
\int_{\kappa_2}^\infty (\Prob^x(\uptau_\cD> s))^{\frac{1}{p'}}\ex^x[\abs{f(X_s)}^p]^{\frac{1}{p}}\, \D{s}
\\
&\leq
(q_{\kappa_2}(0))^{\frac{1}{p}} \norm{f}_{p} \int_{\kappa_2}^\infty (\Prob^x(\uptau_\cD> s))^{\frac{1}{p'}}\, \D{s}
\\
&\leq
(q_{\kappa_2}(0))^{\frac{1}{p}} \norm{f}_{p} \int_{\kappa_2}^\infty s^{-\frac{k}{p'}} \ex^x[\uptau^k]^{\frac{1}{p'}}\, \D{s}
\\
& \leq
(q_{\kappa_2}(0))^{\frac{1}{p}} \norm{f}_{p}\, \frac{p'}{k-p'}\, \kappa_2^{\frac{p'-k}{p'}}
\frac{c_k}{\left(\Psi([\diam \cD]^{-2})\right)^\frac{k}{p'}},
\end{align*}
where in the last line we used Lemma~\ref{L3.1}. This completes the proof.
\end{proof}

\begin{remark}
It is worth pointing out that for the case of the classical Laplacian (i.e., when $\Psi(u)=u$ and ${\underline\mu}=1$), the ABP estimate
in Theorem~\ref{T1.6} is valid for $p>\frac{d}{2}$. This should be compared with \cite{C95}.
\end{remark}

A similar probabilistic approach can be used to obtain a parabolic ABP estimate. Define the parabolic domain to be
$Q_T = [0, T)\times \cD$.
\begin{theorem}[\textbf{Parabolic ABP estimate}]
\label{parabo}
Let $\Psi \in \cB_0$ be strictly increasing and satisfy Assumption~\ref{WLSC}(i). Also, let $V\geq 0$, and $\varphi$ be
a bounded parabolic weak sub-solution of
$$
-\partial_t\varphi + H^{\cD, V}\varphi\leq f \quad \mbox{in}\,\; Q_T,
$$
with $f\in L^p(Q_T)$, for some $p-1>\frac{d}{2{\underline\mu}}$. Then there exists a constant $C = C(p, d, \Psi)$ such that
\begin{equation}\label{ET3.2B}
\sup_{Q_T} \varphi^+
\leq
\left(\sup_{[0, T)\times\cD^c} \varphi^+ \vee \sup_{\{T\}\times\cD} \varphi^+\right)+
C \left[1 + \frac{1}{\left(\Psi([\diam \cD]^{-2})\right)^\frac{2}{p'}}\right] \norm{f}_{p, Q_T},
\end{equation}
where $p, p'$ are H\"older-conjugate exponents.
\end{theorem}
\begin{proof}
From the representation \eqref{ED3.1B} it follows for every $t\in[0, T)$ that
\begin{align}\label{ET3.2A}
\varphi(t, x)
&\leq
\underbrace{\ex^x\left[e^{-\int_0^{(T-t)\wedge\uptau_\cD} V(t+s, X_s)\,\D{s}}
\varphi(t+(T-t)\wedge\uptau_\cD, X_{(T-t)\wedge\uptau_\cD})\right]}_{= \, (i)}
\nonumber
\\
&\,\quad +
\underbrace{\ex^x\left[\int_0^{(T-t)\wedge\uptau_\cD} e^{-\int_0^s V(t+r, X_{r}) \, \D{r}} f(t+s,X_s)\, \D{s}\right]}_
{= \, (ii)}
\end{align}
holds. It is straightforward to see that term (i) in \eqref{ET3.2B} comes from the first term in \eqref{ET3.2A},
as $V\geq 0$. Thus we estimate term (ii) in \eqref{ET3.2A}. We extend $f$ outside of $Q_T$ by $0$, and first suppose
that $T-t\leq \kappa_2$ where $\kappa_2$ is same as in Lemma~\ref{L2.1}. Then
\begin{align*}
(ii) \,
&\leq
\ex^x\left[\int_0^{T-t}  \abs{f(t+s,X_s)}\, \D{s}\right]
= \int_0^{T-t} \int_{\Rd} \abs{f(t+s, y)} q_s(x, y)\, \D{y}\, \D{s}
\\
&\leq
\norm{f}_{p, Q_T} \left[\int_0^{\kappa_2} \int_{\Rd} q^{p'}_s(x, y) \D{y}\, \D{s}\right]^{\frac{1}{p'}}
\leq \kappa_3  \norm{f}_{p, Q_T} \left[\int_0^{\kappa_2} s^{-\frac{d}{2{\underline\mu}(p-1)}} \D{s} \right]^{\frac{1}{p'}}
\\
&=  \kappa_3\,\frac{2{\underline\mu} (p-1)}{2{\underline\mu}(p-1)-d} \,  \norm{f}_{p, Q_T}\, \kappa_2^{\frac{2{\underline\mu}(p-1)-d}{2{\underline\mu} (p-1)}}\,.
\end{align*}
Next suppose $T-t>\kappa_2$. Splitting up the domain of integration, we observe that the rightmost term in
\eqref{ET3.2A} is dominated by
$$
\ex^x\left[\int_0^{\kappa_2}  \abs{f(t+s,X_s)}\, \D{s}\right] + \ex^x\left[\int_{\kappa_2}^{(T-t)\wedge\uptau_\cD}
\abs{f(t+s,X_s)}\, \D{s}\right],
$$
whose first term can be treated as above. Thus we estimate the second term and obtain
\begin{align*}
\ex^x\left[\int_{\kappa_2}^{(T-t)\wedge\uptau_\cD}  \abs{f(t+s,X_s)}\, \D{s}\right]
& =
\ex^x\left[\int_{\kappa_2}^{T-t}\Ind_{\{\uptau_\cD>s\}}  \abs{f(t+s,X_s)}\, \D{s}\right]
\\
&\leq
\left[\int_{\kappa_2}^{T-t}\Prob^x(\uptau_\cD>s)^{\frac{1}{p'}}  \ex^x[\abs{f(t+s,X_s)}^p]^{\frac{1}{p}} \D{s}\right]
\\
&\leq
(q_{\kappa_2}(0))^{\nicefrac{1}{p}} \left[\int_{\kappa_2}^{T-t}\Prob^x(\uptau_\cD>s)^{\frac{1}{p'}}
\norm{f(t+s,\cdot)}_{p, \cD}\,  \D{s}\right]
\\
&\leq
(q_{\kappa_2}(0))^{\nicefrac{1}{p}}\norm{f}_{p, Q_T} \left[\int_{\kappa_2}^{T-t}\Prob^x(\uptau_\cD>s)  \,
\D{s}\right]^{\frac{1}{p'}}
\\
& \leq
(q_{\kappa_2}(0))^{\nicefrac{1}{p}}\norm{f}_{p, Q_T} \left[\int_{\kappa_2}^{\infty}\frac{1}{s^2} \ex^x[\uptau_\cD^2]
\,  \D{s}\right]^{\frac{1}{p'}}
\\
&\leq
\kappa_4 \norm{f}_{p, Q_T} \frac{1}{\left(\Psi([\diam \cD]^{-2})\right)^\frac{2}{p'}},
\end{align*}
where $\kappa_4$ is a suitable constant, and where we used Lemma~\ref{L3.1} in the last step.
\end{proof}

\subsection{Existence and uniqueness of weak solutions}
In this section we use Theorem~\ref{T1.6} to show the existence and uniqueness of weak solutions of the
Dirichlet problem
\begin{equation}
\label{diri}
\left\{\begin{array}{ll}
H^{\cD, V}\varphi = f, \quad \text{in}\, \; \cD \\
\varphi = 0,\, \qquad\quad \; \text{in}\; \cD^c.
\end{array} \right.
\end{equation}

Define the operator
\begin{equation}
\label{semi}
T^{\cD, V}_t f(x)=\ex^x\left[e^{-\int_0^t V(X_s)\, \D{s}} f(X_t)\Ind_{\{t<\uptau_\cD\}}\right], \quad t>0, \, x \in \cD.
\end{equation}
It is shown in \cite[Lem.~3.1]{BL} that $\semi {T^{\cD, V}}$, $T^{\cD, V}_t: L^p(\cD)\to L^p(\cD)$, is a strongly
continuous semigroup, self-adjoint on $L^2(\cD)$, with infinitesimal generator $-H^{\cD, V}$. Moreover, if $V$ is
bounded and $\cD$ satisfies Assumption~\ref{As3.1}, then $T^{\cD, V}_t: \cC(\bar\cD)\to \cC_0(\bar\cD)$, $t > 0$,
and $\semi {T^{\cD, V}}$ is a strongly continuous semigroup.

\begin{proposition}\label{P1.2}
Let $\Psi \in \cB_0$ be strictly increasing and satisfy Assumption~\ref{WLSC}(i). Furthermore, let Assumption~\ref{As3.1}
hold and consider $V, f \in \cC(\bar\cD)$, $V\geq 0$. Then there exists a unique weak-solution $\varphi\in\cC(\Rd)$ of
\eqref{diri}, that is, for every $t\geq 0$ we have for all $x\in\cD$
\begin{equation}\label{EP1.2B}
\varphi(x) = \ex^x\left[e^{-\int_0^{t\wedge\uptau_\cD} V(X_s)\,\D{s}} \varphi(X_{t\wedge\uptau_\cD})\right] +
\ex^x\left[\int_0^{t\wedge\uptau_\cD} e^{-\int_0^s V(X_{r}) \, \D{r}} f(X_s)\, \D{s}\right].
\end{equation}
\end{proposition}

\begin{proof}
Since $V$ is non-negative, $T^{\cD, V}$ is a contraction semigroup. Therefore, for every $\beta>0$ and a given
$f\in\cC(\bar\cD)$, by the Hille-Yosida theorem there exists $\varphi_\beta\in \cC_0(\bar\cD)$ satisfying
$$
(\beta I + H^{\cD, V}) \varphi_\beta = f.
$$
In fact, we have
$$
\varphi_\beta(x) = \int_0^\infty e^{-\beta t} T^{\cD, V}_t f(x)\, \D{t},
$$
and $\varphi_\beta\in \Dom(H^{\cD, V})$. Hence
\begin{align*}
\frac{\D}{\D{t}} \left[e^{-\beta t} T^{\cD, V}_t \varphi_{\beta}\right]
&=
-\beta e^{-\beta t}T^{\cD, V}_t \varphi_\beta + e^{-\beta t} T^{\cD, V}_t (-H^{\cD, V} \varphi_\beta)
\\
&= e^{-\beta t} T^{\cD, V}_t\left(-\beta I -H^{\cD, V}\right) \varphi_\beta
\\
&= - e^{-\beta t} T^{\cD, V}_t f.
\end{align*}
Thus for every $t\geq 0$,
\begin{align*}
\varphi_\beta(x)
&=
e^{-\beta t} T^{\cD, V}_t \varphi_\beta(x) + \int_0^t e^{-\beta s} T^{\cD, V}_s f(x)\, \D{s}
\\
&=
e^{-{\beta} t} T^{\cD, V}_t \varphi_\beta(x) + \int_0^t \ex^x\left[e^{-\int_0^s (\beta+V(X_{r})) \, \D{r}} f(X_s)
\Ind_{\{s<\uptau_\cD\}}\right]\, \D{s}
\\
&= \ex^x\left[e^{-\int_0^{t\wedge\uptau_\cD} (\beta+ V(X_s)\,\D{s})} \varphi_\beta(X_{t\wedge\uptau_\cD})\right]
+ \ex^x\left[\int_0^{t\wedge\uptau_\cD} e^{-\int_0^s (\beta+ V(X_{r})) \, \D{r}} f(X_s)\, \D{s}\right].
\end{align*}
Letting $t\to\infty$, we find
$$
\varphi_\beta(x)= \ex^x\left[\int_0^{\uptau_\cD} e^{-\int_0^s (\beta+ V(X_{r})) \, \D{r}} f(X_s)\, \D{s}\right].
$$
Since $\sup_x\ex^x[\uptau_\cD]<\infty$ and $V\geq 0$, we have that $\varphi_\beta(x)\to\varphi(x)$ as $\beta\to 0$,
where
$$
\varphi(x)= \ex^x\left[\int_0^{\uptau_\cD} e^{-\int_0^s V(X_{r}) \, \D{r}} f(X_s)\, \D{s}\right].
$$
Since $f$ is continuous, it follows that $\varphi\in\cC_0(\bar\cD)$. Thus \eqref{EP1.2B} follows from the strong
Markov property of $\pro X$.
\end{proof}

Now we are ready to prove existence and uniqueness of weak solutions.

\begin{theorem}[\textbf{Existence and uniqueness of weak solution}]
\label{T3.3}
Let Assumptions~\ref{WLSC}(i) and ~\ref{As3.1} hold and $\Psi\in\cB_0$ be strictly increasing. Suppose that $V\geq 0$
is continuous in a neighbourhood of $\cD$, and $f\in L^p(\cD)$ for some $p>\frac{d}{2{\underline\mu}}$. Then there exists a
unique weak-solution $\varphi\in\cC(\Rd)$ of
\eqref{diri}, i.e., for every $t\geq 0$ and $x\in\cD$ we have
\begin{equation}\label{ET3.3A}
\varphi(x) = \ex^x\left[e^{-\int_0^{t\wedge\uptau_\cD} V(X_s)\,\D{s}} \varphi(X_{t\wedge\uptau_\cD})\right] +
\ex^x\left[\int_0^{t\wedge\uptau_\cD} e^{-\int_0^s V(X_{r}) \, \D{r}} f(X_s)\, \D{s}\right].
\end{equation}
\end{theorem}

\begin{proof}
The uniqueness part is obvious from Theorem~\ref{T1.6}, thus we only need to consider the existence part.
Let $V_n, f_n$ denote suitable mollified versions of $V$ and $f$, respectively, so that
$$
\sup_{\cD}\abs{V_n-V}\to 0 \quad \mbox{and} \quad \norm{f_n-f}_{p, \cD}\to 0, \quad \text{as}\;\; n\to\infty.
$$
By Proposition ~\ref{P1.2} there exists $\varphi_n\in\cC(\Rd)$ satisfying \eqref{diri},
that is,
\begin{equation}\label{ET3.3B}
\varphi_n(x) = \ex^x\left[e^{-\int_0^{t\wedge\uptau_\cD} V_n(X_s)\,\D{s}} \varphi_n(X_{t\wedge\uptau_\cD})\right] +
\ex^x\left[\int_0^{t\wedge\uptau_\cD} e^{-\int_0^s V_n(X_{r}) \, \D{r}} f_n(X_s)\, \D{s}\right],
\quad t\geq 0, \; x\in\cD.
\end{equation}
We claim that
\begin{equation}\label{ET3.3C}
\sup_{\Rd}\abs{\varphi_n-\varphi_m}\to 0, \quad \text{as}\; m,n\to\infty.
\end{equation}
It is readily seen from Theorem~\ref{T1.6} that $\sup_{n\in \mathbb N}\norm{\varphi_n}_\infty<\infty$, since
$V_n\geq 0$. Applying Lemma~\ref{L3.1}, we can find $t_0$ such that
$$
\sup_{x\in\cD}\Prob^x(\uptau_\cD>t_0)\leq \frac{1}{2}.
$$
Using the estimate $|e^{a}-e^{b}|\leq \abs{a-b}$, $a, b\leq 0$, we find for every $x\in\cD$
\begin{align*}
& \hspace{-2cm}
\left|\ex^x\left[e^{-\int_0^{t_0\wedge\uptau_\cD} V_n(X_s)\,\D{s}} \varphi_n(X_{t\wedge\uptau_\cD})\right]-
\ex^x\left[e^{-\int_0^{t_0\wedge\uptau_\cD} V_m(X_s)\,\D{s}} \varphi_m(X_{t\wedge\uptau_\cD})\right]\right|
\\
&\leq  \left|\ex^x\left[\left(e^{-\int_0^{t_0\wedge\uptau_\cD} V_n(X_s)\,\D{s}}-
e^{-\int_0^{t_0\wedge\uptau_\cD} V_m(X_s)\,\D{s}} \right)\varphi_n(X_{t\wedge\uptau_\cD})\right]\right|
\\
&\qquad + \ex^x\left[e^{-\int_0^{t_0\wedge\uptau_\cD} V_m(X_s)\,\D{s}}
\left|\varphi_n(X_{t_0\wedge\uptau_\cD})-\varphi_m(X_{t_0\wedge\uptau_\cD})\right| \Ind_{\{t_0<\uptau_\cD\}}\right]
\\
&\leq \norm{\varphi_n}_\infty\, (t_0\norm{V_n-V_m}_{\infty, \cD})
+ \sup_{\Rd} \abs{\varphi_n-\varphi_m} \, \sup_{x\in\cD}\Prob^x(\uptau_\cD>t_0)
\\
&\leq \norm{\varphi_n}_\infty\, (t_0\norm{V_n-V_m}_{\infty, \cD})
+ \frac{1}{2} \sup_{\Rd} \abs{\varphi_n-\varphi_m}.
\end{align*}

We estimate the difference for the second term in \eqref{ET3.3B} to get
\begin{align*}
&\hspace{-1cm}
\left| \ex^x\left[\int_0^{t_0\wedge\uptau_\cD} e^{-\int_0^s  V_n(X_{r}) \, \D{r}}  f_n(X_s)\, \D{s}\right]
- \ex^x\left[\int_0^{t_0\wedge\uptau_\cD} e^{-\int_0^s  V_m(X_{r}) \, \D{r}}  f_m(X_s)\, \D{s}\right]\right|
\\
&\leq \left| \ex^x\left[\int_0^{t_0\wedge\uptau_\cD}\left( e^{-\int_0^s  V_n(X_{r}) \, \D{r}}-
e^{-\int_0^s  V_m(X_{r})}\right)  f_n(X_s)\, \D{s}\right]\right|
\\
&\quad + \ex^x\left[\int_0^{t_0\wedge\uptau_\cD} e^{-\int_0^s  V_m(X_{r}) \, \D{r}}
\left| f_n(X_s)-f_m(X_s)\right|\, \D{s}\right]
\\
&\leq (t_0\norm{V_n-V_m}_{\infty, \cD})\, \ex^x\left[\int_0^{\uptau_\cD}  \abs{f_n(X_s)}\, \D{s}\right]
 + \ex^x\left[\int_0^{\uptau_\cD}   \left| f_n(X_s)-f_m(X_s)\right|\, \D{s}\right]
\\
&\leq \kappa_3 \big((t_0\norm{V_n-V_m}_{\infty, \cD})\, \norm{f_n}_{p,\cD} + \norm{f_n-f_m}_{p,\cD}\big),
\end{align*}
where the last bound follows from the proof of Theorem~\ref{T1.6}. Since $\varphi_n$ vanishes outside $\cD$,
combining the above estimates we obtain
\begin{equation}\label{ET3.3D}
\sup_{\Rd} \abs{\varphi_n-\varphi_m}\leq \kappa_4\, \big( (\norm{f_n}_{p,\cD}+\norm{\varphi_n}_\infty)
\norm{V_n-V_m}_{\infty, \cD} + \norm{f_n-f_m}_{p,\cD} \big)
\end{equation}
for some constant $\kappa_4$ dependent on $\Psi, \cD, p$. This in particular, implies \eqref{ET3.3C}.

Hence we can find a subsequence $\seq \varphi$, denoted in the same way, such that $\norm{\varphi_n-\varphi}_\infty\to 0$
as $n\to\infty$ and for some $\varphi\in\cC_0(\cD)$. It is also easy to see that we can pass to the limit in the first
term at the right hand side of \eqref{ET3.3C}, while to take limit in the rightmost term we may employ a similar argument
as above. Thus we obtain
\begin{equation*}
\varphi(x) = \ex^x\left[e^{-\int_0^{t\wedge\uptau_\cD} V(X_s)\,\D{s}} \varphi(X_{t\wedge\uptau_\cD})\right] +
\ex^x\left[\int_0^{t\wedge\uptau_\cD} e^{-\int_0^s V(X_{r}) \, \D{r}} f(X_s)\, \D{s}\right],
\end{equation*}
which gives \eqref{ET3.3A}.
\end{proof}

\begin{remark}[\textbf{Viscosity solutions}]
\label{R3.1}
The weak (semigroup) solution, as defined above, is related to the viscosity solution used for non-local equations.
 Let $H_0=\Psidel$.
Then for $f, V$  continuous,  our weak solution is actually a viscosity solution. This can be seen
as follows. Let $x\in\cD$ and $\sB_\delta(x)$ be the ball of radius $\delta$ around $x$ so that $\sB_\delta(x)
\subset\cD$. Also, denote by $\uptau_\delta$ the first exit time from $\sB_\delta(x)$. Then using the strong
Markov property of subordinate Brownian motion, it follows that for $t\geq 0$
\begin{equation}\label{ER1.1B}
\varphi(x) = \ex^x\left[e^{-\int_0^{t\wedge\uptau_\delta} V(X_s)\,\D{s}} \varphi(X_{t\wedge\uptau_\delta})\right]
+ \ex^x\left[\int_0^{t\wedge\uptau_\delta} e^{-\int_0^s V(X_{r}) \, \D{r}} f(X_s)\, \D{s}\right].
\end{equation}
Let $\psi \geq \varphi$ be a bounded continuous test function such that $\varphi-\psi$ has a global maximum $0$ at 
$x$ and $\psi\in \cC^2(\sB_{2\delta})$ for some $r$. We may modify $\psi$ to $\varphi$ outside $\sB_{2\delta}$. Thus 
by It\^o's formula it is readily seen that
\begin{align*}
\ex^x\left[\int_0^{t\wedge\uptau_\delta} e^{-\int_0^s V(X_{r})\, \D{r}} H^{\cD, V}\psi(X_s) \D{s}\right]
&= \psi(x) - \ex^x\left[e^{-\int_0^{t\wedge\uptau_\delta} V(X_s)\,\D{s}} \psi(X_{t\wedge\uptau_\delta})\right]
\\
&\leq \varphi(x) - \ex^x\left[e^{-\int_0^{t\wedge\uptau_\delta} V(X_s)\,\D{s}} \varphi(X_{t\wedge\uptau_\delta})\right]
\\
&= \ex^x\left[\int_0^{t\wedge\uptau_\delta} e^{-\int_0^s V(X_{r}) \, \D{r}} f(X_s)\, \D{s}\right],
\end{align*}
where the last line follows from \eqref{ER1.1B}. Finally, divide both sides by $t$ and let $t\to 0$, to conclude that
$$
H^{\cD, V}\psi(x) \leq f(x).
$$
Hence $\varphi$ is a viscosity sub-solution at $x$. Similarly, we can verify the viscosity super-solution property;
see, for instance, \cite{CS-09}.
For $f\in L^p$, our notion of weak solution can be related to the $L^p$-viscosity solution in \cite{CCKS}.
\end{remark}

\begin{remark}\label{R3.2}
The technique of Theorem \ref{T1.6} is not restricted to subordinate Brownian motion and can be used for a more
general class of Markov processes. For instance, we may consider the stable-like operator
$$
\mathcal{A}_\alpha \varphi(x)  =\mathrm{PV} \int_{\Rd}(\varphi(x)-\varphi(x+y)) \frac{C(x, y)}{\abs{y}^{d+\alpha}},
\quad 0<\alpha<2.
$$
Here $C(x, y)=C(x, -y)$ is assumed to be bounded from above and below. If $C$ is assumed to be H\"{o}lder continuous
in its first argument, then by \cite{CZ16} it is known that there exists a heat kernel associated with the above
generator, which gives rise to a strong Feller process and the corresponding transition density $p_C(t,x,y)$ has a
bound similar to the right hand side of \eqref{EL2.1A}, with ${\underline\mu}=\alpha/2$. A bound like \eqref{AB100} can be
obtained from the Chapman-Kolmogorov equality
$$
p_C(t+s, x, y) = \int_{\Rd}p_C(t,x,z)p_C(s, z, y) \, \D{z}\,.
$$
It can be seen from the above that $\sup_{t\geq 1} p_C(t, x, y)\leq \sup_{t\in[\frac{1}{2}, 1]} p_C(t, x, y)$.
On the other hand, an estimate similar to \eqref{EL3.1A} also holds (see \cite[Lem.~2.1]{ABC16}). For a
uniqueness-related discussion of the solutions we refer to \cite{MP14, CZ17}. Thus it is possible to obtain an
ABP-type estimate for Markov processes associated with the above generator. Due to a similar reason, the methodology
of Theorem~\ref{T1.6} is also applicable for diffusion processes.
\end{remark}

From the proof of Theorem~\ref{T3.3} (see \eqref{ET3.3D}) we also obtain the following stability result.
\begin{theorem}[\textbf{Continuous dependence}]
Let Assumptions~\ref{WLSC}(i) and ~\ref{As3.1} hold, and $\Psi\in\cB_0$ be strictly increasing. Suppose that $V_1, V_2
\geq 0$ are continuous on $\bar\cD$, and $f_1,f_2\in L^p(\cD)$ for some $p>\frac{d}{2{\underline\mu}}$. Denote by $\varphi_1,
\varphi_2 \in\cC(\Rd)$ the weak solutions of \eqref{diri} corresponding to the given coefficients, respectively.
Then there exists a constant $C_1> 0$, dependent on $\cD, \Psi, p$, such that
\begin{equation*}
\norm{\varphi_1-\varphi_2}_{\infty}\leq C_1\Bigl( (\norm{\varphi_1}_\infty+\norm{\varphi_2}_\infty +
\norm{f_1}_p+\norm{f_2}_p)\,\norm{V_1-V_2}_{\infty, \cD} + \norm{f_1-f_2}_p\Bigr).
\end{equation*}
\end{theorem}

\bigskip

\section{\textbf{Maximum principles for non-local Schr\"odinger operators}}
\subsection{Stochastic representation of the principal eigenfunction}
The main result of this section is a stochastic representation for the principal eigenfunction (Theorem~\ref{T1.1}
below) of the non-local Schr\"{o}\-din\-ger operator with Dirichlet exterior condition, which solves
\begin{equation*}
\label{dirisch*}
\left\{\begin{array}{ll}
H^{\cD, V}\varphi = \lambda \varphi, \quad \text{in}\, \; \cD \\
\varphi = 0,\, \qquad\quad \;\;\;\, \text{in}\; \cD^c.
\end{array} \right.
\end{equation*}
We will use the following assumption in this section.
\begin{assumption}\label{As4.1}
There exists a family of decreasing bounded domains $\{\cD_n\}_{n\geq 1}$ such that $\cD_{n+1}\Subset\cD_{n}$ and
$\cap_{n\geq 1}\cD_n=\cD$. Moreover, all $\cD_n$ and $\cD$ have regular boundary in the sense of Assumption~\ref{As3.1}.
\end{assumption}
\noindent
In view of \cite[Lem.~2.9]{BGR15} we note that any domain $\cD$ which is convex or has a $\cC^2$ boundary, satisfies 
Assumption~\ref{As4.1}.

Recall the semigroup \eqref{semi}.
The results in the lemma below have been obtained in \cite[Lem.~3.1]{BL}, see a proof there.
\begin{lemma}\label{L1.1}
Let $\Psi \in \cB_0$ and $V\in L^\infty(\Rd)$. Also, let $\Psi$ satisfy the Hartman-Wintner condition \eqref{HW}.
The following properties hold.
\begin{itemize}
\item[(i)]
Every $T^{\cD, V}_t$, $t > 0$, is an integral operator with symmetric kernel $T^{\cD, V}(t,x,y)$, that is,
$$
T^{\cD, V}_t f(x) = \int_{\cD} T^{\cD, V}(t, x, y) f(y) \, \D{y}, \quad x \in \Rd\,,
$$
where
\begin{equation}
\label{EL1.2A}
T^{\cD, V}(t, x, y)= \ex^0_{\Prob_S} \left[p_{S^\Psi_t}(x-y) \ex^{x,y}_{0,S^\Psi_t}\left[e^{-\int_0^t V(B_{S^\Psi_s})ds}
\Ind_{\{\uptau_D > t\}}\right]\right],
\end{equation}
$p_t(x) = (4\pi t)^{-d/2}e^{-\frac{|x|^2}{4t}}$, and $\ex^{x,y}_{0,S^\Psi_t}$ denotes expectation with respect to the
Brownian bridge measure from $x$ at time 0 to $y$ at time $s$, evaluated at random time $s=S^\Psi_t$.
\item[(ii)]
$\{T^{\cD, V}_t: t\geq 0\}$ is a strongly continuous semigroup on $L^2(\cD)$, with generator $-H^{\cD, V}=-\Psidel - V$.
\item[(iii)]
 $T^{\cD, V}_t$ is a Hilbert-Schmidt
operator on $L^2(\cD)$, for every $t>0$.
\end{itemize}
\end{lemma}
\noindent
As well-known, see a discussion in \cite{BL}, the principal eigenfunction $\varphi^*$ of $H^{\cD, V}$ is strictly positive
and the corresponding eigenvalue $\lambda^*$ is simple. We note that the assumption of $V$ being bounded is not necessary,
and our conclusions can be extended to Kato-class as in the same reference. Moreover, under Assumption~\ref{As3.1} it follows
that $\varphi^* \in \cC_0(\cD)$. For a given family $\{\cD_n\}_{n\geq 1}$, as required in Assumption~\ref{As4.1}, we denote
the associated principal eigenpair by $(\lambda^*_n, \varphi^*_n)$.

\begin{lemma}\label{L1.2}
Assume that $\Psi \in \cB_0$ satisfies the Hartman-Wintner property \eqref{HW}. The following hold.
\begin{itemize}
\item[(i)]
For every $n\in\NN$ we have $\lambda^*>\lambda^*_n$. Moreover, $\lim_{n\to \infty}\lambda^*_n =\lambda^*$.
\item[(ii)]
Let $\widetilde{V}\geq V$ and suppose that for an open set $\sU\subset\cD$ we have $\widetilde{V}> V$ in $\sU$. Then
$ \lambda^*_{\widetilde{V}}> \lambda^*_V$, where $\lambda^*_V$ and $\lambda^*_{\widetilde V}$ denote the principal
eigenvalues corresponding to the potentials $V$ and $\widetilde{V}$, respectively.
\end{itemize}
\end{lemma}

\begin{proof}
(i)\; First we prove domain monotonicity of the principal eigenvalue. Note that
$$
T_t^{\cD, V} \varphi^* = e^{-\lambda^* t} \varphi^*, \quad t\geq 0.
$$
The same relation holds for $\cD_n, \lambda^*_n, \varphi^*_n$. Due to self-adjointness of $T_t^{\cD, V}$, we have
$$
e^{-\lambda^* t}=\max \left\{\langle T_t^{\cD, V} \xi, \xi\rangle\; \colon\;  \norm{\xi}_2=1\right\}.
$$
Since $L^2(\cD)\subset L^2(\cD_n)$ (on extension by $0$ to the larger domain), it is obvious that $\lambda^*_n\leq
\lambda^*$. Suppose that $\lambda^*_n= \lambda^*$.
Using the expression \eqref{EL1.2A} from Lemma~\ref{L1.1}(i), we see that $T^{\cD_n, V}(t, x, y)\geq
T^{\cD, V}(t, x, y)$, for all $x, y \in \cD$. We normalize by $\norm{\varphi^*}_2 =1$ and extend $\varphi^*$ by
$0$ outside $\cD$. Note that if
$$
T^{\cD_n, V}_t \varphi^*>e^{-\lambda^* t}\varphi^* \quad \text{in a set of positive Lebesgue measure in}\; \cD,
$$
then $\langle T^{\cD_n, V}_t \varphi^*, \varphi^*\rangle > e^{-\lambda^* t}$, leading to a contradiction for
$\lambda^*_n= \lambda^*$. On the other hand,
$$
T^{\cD_n, V}_t \varphi^*=e^{-\lambda^* t}\varphi^* \quad \text{in}\; \cD
$$
implies
$$
\langle T^{\cD_n, V}_t \varphi^*, \varphi^*\rangle=e^{-\lambda^* t},
$$
and hence $\varphi^*$ is an eigenfunction corresponding to the eigenvalue $e^{-\lambda^* t}$. By uniqueness of the
principal eigenfunction it follows that $\varphi^*$ can only be a positive multiple of $\varphi^*_n$. However, this
is not possible as $\varphi^*_n$ is strictly positive in $\cD_n$. Thus $\lambda^*_n = \lambda^*$ is not possible
and hence $\lambda^*_n < \lambda^*$ holds.

Now we show that $\lim_{n\to \infty}\lambda^*_n =\lambda^*$. Suppose that $\lim_{n\to \infty}
\lambda^*_n =\lambda_0$. We extend $\varphi^*_n$ by $0$ in $\cD_1\setminus\cD_{n}$. Normalizing by $\norm{\varphi^*_n}_2
=1$, we may also assume that
$$
\varphi^*_n\rightharpoonup \varphi_0 \quad \text{in}\; L^2(\cD_1) \quad \text{as} \quad n\to \infty,
$$
for some $\varphi_0\in L^2(\cD_1)$. It is also easily seen that $\varphi_0$ is supported on $\bar\cD$. We show that
for every $x\in \cD$
\begin{equation}\label{EL1.2B}
T^{\cD_{n}, V}(t, x, \cdot)\to T^{\cD, V}(t, x, \cdot) \quad \text{in}\; L^2(\cD_1).
\end{equation}
Since $\uptau_{\cD_n}=\uptau_{\bar{\cD_n}}$ by Assumption~\ref{As4.1}, we obtain from \eqref{EL1.2A} that
$T^{\cD_{n}, V}(t, x, y)=0$ for every $y\in \cD^c_{n}$. Again, $\uptau_{\cD}=\uptau_{\bar\cD}$ $\Prob$-a.s. implies
that
$$
\uptau_{\cD_n}\searrow \uptau_\cD \quad \text{$\Prob$-a.s., \, as $n\to \infty$}.
$$
Also, note that for $y\in\cD$ we have $\{\uptau_\cD=t, X_t=y\}=\emptyset$. Thus for every fixed $y\in\cD$,
\begin{align*}
\ex^0_{\Prob_S}\left[p_{S^\Psi_t} (x-y) \ex^{x, y}_{0, S^\Psi_t} \left[e^{-\int_0^t V(B_s)\, \D{s}}
\Ind_{\{\uptau_{\cD_{n}}>t\}}\right]\right]
\to \ex^0_{\Prob_S}\left[p_{S^\Psi_t} (x-y) \ex^{x, y}_{0, S^\Psi_t} \left[e^{-\int_0^t V(B_s)\, \D{s}}
\Ind_{\{\uptau_\cD> t\}}\right]\right].
\end{align*}
For $t>0$ and $V$ is bounded, $T^{\cD_{n}, V}(t, x, y)$ is uniformly bounded by $\ex^0_{\Prob_S}[p_{S^\Psi_t}(x-y)]
e^{t\norm{V}_\infty}=q_t(x-y)e^{t\norm{V}_\infty}$, where $q$ is the transition probability density of subordinate
Brownian motion, which itself is bounded due to the Hartman-Wintner condition. Thus by dominated convergence (using
$\abs{\partial\cD}=0$), we obtain \eqref{EL1.2B} as $\cD$ is bounded. This implies
$$
\int_{\cD_1} T^{\cD_{n}, V}(t, x, y)\varphi^*_n(y)\, \D{y}\to \int_{\cD_1} T^{\cD, V}(t, x, y)\varphi_0(y)\, \D{y},
\quad x\in\cD.
$$
Since for every $x\in\cD$ we have
$$
\int_{\cD_1} T^{\cD, V}(t, x, y)\varphi_0(y)\,
\D{y}=\int_{\cD} T^{\cD, V}(t, x, y)\varphi_0(y)\, \D{y},
$$
we get $\varphi^*_n\to\varphi_0$ pointwise for $x\in\cD$. It is again direct to see that $\norm{\varphi^*_n}_\infty$ is
uniformly bounded above in $n$. Thus by the dominated convergence theorem we obtain $\norm{\varphi^*_n-\varphi_0}_{2, \cD}
\to 0$. This, in particular, implies $\varphi_0\gneq 0$. Hence we can take the limit in $T^{\cD, V}_t \varphi^*_n(x) =
e^{-\lambda^*_n t}\varphi^*_n$ to obtain
$$
T^{\cD, V}_t \varphi_0(x) = e^{-\lambda_0 t}\varphi_0(x),
$$
which also gives $\varphi_0 > 0$ in $\cD$. By uniqueness of a positive eigenfunction it then follows that $\lambda_0=\lambda^*$.
This completes the proof of part (i).

Next we prove (ii). Clearly, $\lambda^*_V \leq \lambda^*_{\widetilde{V}}$. Let $\varphi^*_{V}$ and $\varphi^*_{\widetilde V}$
denote the principal eigenfunctions corresponding to $\lambda^*_V$ and $\lambda^*_{\widetilde{V}}$, respectively. From
\eqref{EL1.2A} it is seen that for every $y\in \sU$
$$
T^{\cD, V}(t, x, y)> T^{\cD, \widetilde{V}}(t, x, y).
$$
This follows from the fact that the subordinator $\pro{S^\Psi}$ jumps at time $t$ with probability zero. Therefore,
on normalizing in $L^2$, we obtain
\begin{align*}
e^{-\lambda^*_{\widetilde{V}} t} & = \langle T^{\cD, \widetilde{V}}_t \varphi^*_{\widetilde{V}}, \varphi^*_{\widetilde{V}}\rangle
= \int_{\cD}\int_{\cD} T^{\cD, \widetilde{V}}(t, x, y)\varphi^*_{\widetilde{V}}(x) \varphi^*_{\widetilde{V}}(y)\, \D{y}\, \D{x} \\
& < \langle T^{\cD, {V}}_t \varphi^*_{\widetilde{V}}, \varphi^*_{\widetilde{V}}\rangle
\leq \max\left\{ \langle T^{\cD, V}_t \xi, \xi \rangle\; :\; \norm{\xi}_2=1\right\}
= e^{-\lambda^*_V t}.
\end{align*}
This gives $\lambda^*_{\widetilde V}> \lambda^*_V$.
\end{proof}

As an immediate consequence of Lemma ~\ref{L1.1} we have the following result.
\begin{corollary}
Let $\Psi \in \cB_0$ satisfy the Hartman-Wintner property \eqref{HW}, and let Assumption~\ref{As4.1} hold. Then
$$
\lambda^*
= -\lim_{t\to\infty} \frac{1}{t}\, \log\, \ex^x\left[e^{-\int_0^t V(X_s)\, \D{s}}\Ind_{\{\uptau_\cD> t\}}\right].
$$
\end{corollary}

\begin{proof}
Since $\varphi^*\in\cC_0(\cD)$, we have
$$
e^{-\lambda^* t} \varphi^*(x) = \ex^x\left[e^{-\int_0^t V(X_s)\, \D{s}} \varphi^*(X_t)\Ind_{\{\uptau_\cD> t\}}\right]
\leq \norm{\varphi^*}_\infty \ex^x\left[e^{-\int_0^t V(X_s)\, \D{s}}\Ind_{\{\uptau_\cD> t\}}\right].
$$
By taking logarithms on both sides and dividing by $t$, we obtain
\begin{equation}\label{EC4.1A}
\lambda^*\geq -\liminf_{t\to\infty} \frac{1}{t}\, \log\, \ex^x\left[e^{-\int_0^t V(X_s)\, \D{s}}\Ind_{\{\uptau_\cD> t\}}\right].
\end{equation}
On the other hand, from Lemma~\ref{L1.1} we have $\lambda^*_n<\lambda^*$ and $\varphi^*_n>0$ in $\cD_n$. Thus
\begin{align*}
e^{-\lambda^*_nt} \varphi^*(x) &= \ex^x\left[e^{-\int_0^t V(X_s)\, \D{s}} \varphi^*_n(X_t)\Ind_{\{\uptau_{\cD_n}> t\}}\right]
\\
& \geq  \ex^x\left[e^{-\int_0^t V(X_s)\, \D{s}} \varphi^*_n(X_t)\Ind_{\{\uptau_{\cD}> t\}}\right] \\
& \geq \min_{\cD}\varphi^*_n\, \ex^x\left[e^{-\int_0^t V(X_s)\, \D{s}} \Ind_{\{\uptau_{\cD}> t\}}\right].
\end{align*}
From here we find
$$
\lambda^*_n \leq
-\limsup_{t\to\infty} \frac{1}{t}\, \log\, \ex^x\left[e^{-\int_0^t V(X_s)\, \D{s}}\Ind_{\{\uptau_\cD> t\}}\right].
$$
Letting $n\to\infty$ and using Lemma ~\ref{L1.2}(i), we furthermore obtain
\begin{equation}\label{EC4.1B}
\lambda^*
\leq -\limsup_{t\to\infty} \frac{1}{t}\, \log\, \ex^x\left[e^{-\int_0^t V(X_s)\, \D{s}}\Ind_{\{\uptau_\cD> t\}}\right].
\end{equation}
Combining \eqref{EC4.1A} and \eqref{EC4.1B} gives the result.
\end{proof}

Our next main result is a stochastic representation of the principal eigenfunction. As above, we denote the principal
eigenpair corresponding to the Schr\"{o}dinger operator $H^{\cD, V}$ by $(\varphi^*, \lambda^*)$.

\begin{theorem}[\textbf{Stochastic representation of principal eigenfunction}]
\label{T1.1}
Let the conditions in Lemma~\ref{L1.1} hold together with Assumption~\ref{As4.1}. Consider any point $\hat{x}\in\cD$ and
denote by $\sB_r$ the ball of radius $r$ around it. Then we have
\begin{equation}\label{ET1.1A}
\varphi^*(x) =
\ex^x\left[e^{\int_0^{\Breve\uptau_r} (\lambda^*-V(X_s))\, \D{s}}
\varphi^*(X_{\Breve\uptau_r})\Ind_{\{\Breve\uptau_r<\uptau_\cD\}}\right], \quad x\in\cD\setminus \bar\sB_r,
\end{equation}
where $\Breve\uptau_r$ denotes the first hitting time of the ball $\sB_r$ by $\pro X$. In particular, we have for
$x\in\cD$ that
\begin{equation}\label{ET1.1B}
\varphi^*(x) = \varphi^*(\hat{x})\, \lim_{r\to 0} \ex^x\left[e^{\int_0^{\Breve\uptau_r} (\lambda^*-V(X_s))\, \D{s}}
\Ind_{\{\Breve\uptau_r<\uptau_\cD\}}\right]\,.
\end{equation}
\end{theorem}

\begin{proof}
Fix $r$ small enough so that $\sB_r\Subset \cD$ holds. Recall that $\varphi^*$ is strictly positive  in $\cD$ and
continuous on $\bar\cD$. From the strong Markov property of $\pro X$ it is immediate that
\begin{equation}\label{ET1.1C}
\varphi^*(x) = \ex^x\left[e^{\int_0^{t\wedge \Breve\uptau_r} (\lambda^*-V(X_s))\, \D{s}}
\varphi^*(X_{t\wedge \Breve\uptau_r})\Ind_{\{t\wedge \Breve\uptau_r<\uptau_\cD\}}\right]
\end{equation}
Thus by letting $t\to\infty$ and applying Fatou's lemma in \eqref{ET1.1C}, we get
\begin{equation}\label{ET1.1D}
\varphi^*(x) \geq  \ex^x\left[e^{\int_0^{\Breve\uptau_r} (\lambda^*-V(X_s))\, \D{s}}
\varphi^*(X_{\Breve\uptau_r})\Ind_{\{\Breve\uptau_r<\uptau_\cD\}}\right].
\end{equation}
Note that \eqref{ET1.1D} also implies
$$
\limsup_{r\to 0}\, \ex^x\left[e^{\int_0^{\Breve\uptau_r} (\lambda^*-V(X_s))\, \D{s}}
\Ind_{\{\Breve\uptau_r<\uptau_\cD\}}\right]\leq \frac{\varphi^*(x)}{\varphi^*(\hat x)}.
$$
Define
\[
\widetilde{V}=\left\{\begin{array}{ll}
V & \text{for}\; x\in \sB_r \cap \cD,
\\
\norm{V}_\infty + 1 & \text{for}\; x\in \sB_r.
\end{array}
\right.
\]
Then by Lemma~\ref{L1.2}(ii) we have $\lambda^*_{\widetilde{V}} > \lambda^*$. Using the domain continuity property
from Lemma~\ref{L1.2}(i), we can find $n$ large enough such that $\widetilde\lambda^*_n>\lambda^*$, where $(\widetilde
\lambda^*_n, \widetilde\varphi^*_n)$ is the principal eigenpair in $\cD_n$ with potential $\widetilde{V}$. Also, note
that $\widetilde \varphi^*_n$ is strictly positive in $\cD_n$. Therefore, since $V=\widetilde V$ on $\sB_r^c$, we get
\begin{align*}
& \ex^x\left[e^{\int_0^{t} (\lambda^*-V(X_s))\, \D{s}} \varphi^*(X_{t})\Ind_{\{t<\uptau_\cD\}}
\Ind_{\{t\leq \Breve\uptau_r\}}\right]
\\
& =
e^{(\lambda^*-\widetilde\lambda^*_n) t}\ex^x\left[e^{\int_0^{t} (\widetilde\lambda^*_n-\widetilde{V}(X_s))\, \D{s}}
\varphi^*(X_{t})\Ind_{\{t<\uptau_\cD\}} \Ind_{\{t\leq \Breve\uptau_r\}}\right]
\\
&\leq
e^{(\lambda^*-\widetilde\lambda^*_n) t} \frac{\max_{\cD} \varphi^*}{\min_{\cD}\varphi^*_n}
\ex^x\left[e^{\int_0^{t} (\widetilde\lambda^*_n-\widetilde{V}(X_s))\, \D{s}} \widetilde\varphi^*_n(X_{t})
\Ind_{\{t<\uptau_\cD\}} \Ind_{\{t\leq \Breve\uptau_r\}}\right]
\\
&\leq
e^{(\lambda^*-\widetilde\lambda^*_n) t} \frac{\max_{\cD} \varphi^*}{\min_{\cD}\widetilde\varphi^*_n}
\ex^x\left[e^{\int_0^{t} (\widetilde\lambda^*_n-\widetilde{V}(X_s))\, \D{s}} \widetilde\varphi^*_n(X_{t})
\Ind_{\{t<\uptau_{\cD_n}\}}
\Ind_{\{t\leq \Breve\uptau_r\}}\right]
\\
&\leq e^{(\lambda^*-\widetilde\lambda^*_n) t} \frac{\max_{\cD} \varphi^*}{\min_{\cD}\widetilde\varphi^*_n}
\widetilde\varphi^*_n (x),
\end{align*}
where in the last line we used \eqref{ET1.1C} for the eigenpair $(\widetilde\lambda^*_n, \widetilde\varphi^*_n)$.
Hence by letting $t\to \infty$ in the above expression, we see that
$$
\ex^x\left[e^{\int_0^{t} (\lambda^*-V(X_s))\, \D{s}} \varphi^*(X_{t})\Ind_{\{t<\uptau_\cD\}}
\Ind_{\{t\leq \Breve\uptau_r\}}\right]\to 0.
$$
Next using \eqref{ET1.1D} and the monotone convergence theorem in \eqref{ET1.1C}, we find that
$$
\varphi^*(x) = \ex^x\left[e^{\int_0^{\Breve\uptau_r} (\lambda^*-V(X_s))\, \D{s}}
\varphi^*(X_{\Breve\uptau_r})\Ind_{\{\Breve\uptau_r<\uptau_\cD\}}\right],
$$
which proves \eqref{ET1.1A}. Equality \eqref{ET1.1B} follows by \eqref{ET1.1A} and the continuity of $\varphi^*$.
\end{proof}

\subsection{Maximum principles}

A first consequence of Theorem~\ref{T1.1} is a refined maximum principle in the sense of the classic result
\cite[Prop.~6.2]{BNV}. Recall from Definition~\ref{D3.1} that a continuous function $w$ is a weak super-solution of
\begin{equation}\label{E4.9}
H^{\cD, V} w\geq 0 \quad \text{in}\; \cD, \quad \mbox{and} \quad w=0\,\; \text{in}\; \cD^c,
\end{equation}
if
$$
T^{\cD, V}_t w(x)\leq w(x), \quad x\in\cD, \; t>0,
$$
holds. Also, the function $w$ is said to be a weak sub-solution of \eqref{E4.9} if $-w$ is a weak super-solution,
and $w$ is a weak solution if it is both a weak sub- and super-solution. Note that $0$ is always a weak solution to
the above problem. Recall the notation
$$
h \gneq 0 \quad \mbox{meaning} \quad h(x) \geq 0 \; \mbox{for all $x \in \cD$ and $h \not\equiv 0$.}
$$
\begin{theorem}[\textbf{Refined maximum principle}]
\label{T1.2}
Suppose that $\Psi \in \cB_0$ satisfies the Hartman-Wintner property \eqref{HW} and Assumption~\ref{As4.1} holds.
Let $w_1$ be a weak-supersolution and $w_2$ be a weak-subsolution of \eqref{E4.9}. Furthermore, assume that
$\lambda^*>0$. Then we have either $w_1= w_2$ or $w_1>w_2$ in $\cD$.
\end{theorem}

\begin{proof}
Since $-w_2$ is a weak super-solution and the addition of two super-solutions is again a super-solution, it suffices
to show that if $w$ is a super-solution, then $w\geq 0$. First notice that $w\lneq 0$ is not possible, since otherwise
for $\norm{w}_{2, \cD}=1$ we would have
$$
1 \leq \langle T^{\cD, V}_t w, w\rangle \leq e^{-\lambda^* t}, \quad t>0,
$$
contradicting that $\lambda^*> 0$. Thus either $w=0$ or $w^+ > 0$ holds at some point in $\cD$. We show that if $w^+$
is positive at some point of $\cD$, then it is positive everywhere in $\cD$. Suppose that $w(\hat{x})>0$ for some
$\hat{x}\in \cD$. We show that
\begin{equation}\label{ET1.2A}
w(x) \geq \ex^x\left[e^{-\int_0^{\Breve\uptau_r} V(X_s)\, \D{s}}
w(X_{\Breve\uptau_r})\Ind_{\{\Breve\uptau_r<\uptau_\cD\}}\right], \quad x\in\cD\setminus \sB_r.
\end{equation}
This will imply that $w>0$ in $\cD$, proving the theorem. Thus it remains to show \eqref{ET1.2A}. Note that
$\left(e^{-\int_0^t V(X_s) \D{s}} w(X_t)\Ind_{\{t<\uptau_\cD\}}\right)_{t\geq 0}$ is a super-martingale
with respect to the natural filtration $\pro \cF$ of $\pro X$. Indeed, taking $s<t$ and using the Markov property
of $\pro X$, we see that
\begin{align*}
\ex\left[e^{-\int_0^t V(X_r) \D{r}} w(X_t)\Ind_{\{t<\uptau_\cD\}} \Big| \cF_s \right]
&=
\ex\left[e^{-\int_0^s V(X_r) \D{r}} e^{-\int_s^t V(X_r) \D{r}} w(X_t)\Ind_{\{s<\uptau_\cD \}}
\Ind_{\{t<\uptau_\cD\}} \Big| \cF_s \right]
\\
&=
e^{-\int_0^s V(X_r) \D{r}}\Ind_{\{s<\uptau_\cD \}} \ex^{X_s}\left[ e^{-\int_0^{t-s} V(X_r) \D{r}} w(X_{t-s})
\Ind_{\{t-s<\uptau_\cD\}}\right]
\\
&\leq e^{-\int_0^s V(X_r) \D{r}}\Ind_{\{s<\uptau_\cD \}} w(X_s),
\end{align*}
where the last inequality follows from the definition of super-solution. Thus by the optional sampling theorem we have
\begin{align}\label{ET1.2B}
w(x) & \geq \ex^x\left[e^{-\int_0^{t\wedge \Breve\uptau_r} V(X_s)\, \D{s}}
w(X_{t\wedge \Breve\uptau_r})\Ind_{\{t\wedge \Breve\uptau_r<\uptau_\cD\}}\right]\nonumber
\\
& = \ex^x\left[e^{-\int_0^{t} V(X_s)\, \D{s}} w(X_{t})\Ind_{\{t\leq \Breve\uptau_r \}}\Ind_{\{t<\uptau_\cD\}}\right]
+ \ex^x\left[e^{-\int_0^{\Breve\uptau_r} V(X_s)\, \D{s}} w(X_{\Breve\uptau_r})\Ind_{\{\Breve\uptau_r < t \}}
\Ind_{\{\Breve\uptau_r<\uptau_\cD\}}\right].
\end{align}
By Lemma~\ref{L1.1}(a) we can find $n$ large enough such that $\lambda^*_n>0$. Using the stochastic representation
\eqref{ET1.1C} of $\varphi^*_n$, we obtain
\begin{align*}
& \hspace{-1cm}
\ex^x\left[e^{-\int_0^{t} V(X_s)\, \D{s}} w(X_{t})\Ind_{\{t\leq \Breve\uptau_r \}}\Ind_{\{t<\uptau_\cD\}}\right]
\\
& \leq e^{-\lambda^*_n t} \frac{\max_\cD w}{\min_{\cD} \varphi^*_n} \,
\ex^x\left[e^{\int_0^{t} (\lambda^*_n-V(X_s))\, \D{s}} \varphi^*_n(X_{t})
\Ind_{\{t\leq \Breve\uptau_r \}}\Ind_{\{t<\uptau_\cD\}}\right]
\\
& \leq e^{-\lambda^*_n t} \frac{\max_\cD w}{\min_{\cD} \varphi^*_n} \,
\ex^x\left[e^{\int_0^{t} (\lambda^*_n-V(X_s))\, \D{s}} \varphi^*_n(X_{t})
\Ind_{\{t\leq \Breve\uptau_r \}}\Ind_{\{t<\uptau_{\cD_n}\}}\right]
\\
& \leq e^{-\lambda^*_n t} \frac{\max_\cD w}{\min_{\cD} \varphi^*_n}\varphi^*_n (x)\to 0 \quad \mbox{as $t\to\infty$,}
\end{align*}
where in the third line we used \eqref{ET1.1C}. Thus by letting $t\to\infty$ in \eqref{ET1.2B} and applying the
monotone convergence theorem, we obtain \eqref{ET1.2A}.
\end{proof}

A converse of Theorem~\ref{T1.2} also holds.
\begin{theorem}
\label{conv}
If for any super-solution $w\in\cC(\Rd)$ of
$$
H^{\cD, V} w\geq 0 \;\; \text{in}\; \; \cD, \quad \mbox{and} \quad w=0\, \;\;\text{in}\;\; \cD^c,
$$
we have $w\geq 0$, then $\lambda^*>0$.
\end{theorem}

\begin{proof}
Suppose, to the contrary, that $\lambda^*\leq 0$. Let $\varphi^*$ be the (strictly positive) principal eigenfunction.
Then we know that
$$
e^{\lambda^* t} T^{\cD, V}_t \varphi^*(x) =\varphi^*(x), \quad  t>0\,,
$$
and therefore
$$
T^{\cD, V}_t \varphi^*(x) \geq \varphi^*(x), \quad t>0\,,$$
implying
$$
T^{\cD, V}_t (-\varphi^*)(x) \leq (-\varphi^*)(x), \quad t>0\,.
$$
Thus $-\varphi^*$ is a weak super-solution. However, it is negative in $\cD$, which contradicts the assumption.
Hence $\lambda^*>0$.
\end{proof}

Next we derive a uniqueness result from the stochastic representation of $\varphi^*$.
\begin{proposition}\label{P1.1}
Let $\psi\in\cC(\Rd)$ be a positive weak super-solution of
$$
H^{\cD, V}\psi -\lambda \psi \geq 0, \quad \mbox{with} \quad \psi=0 \, \; \text{in}\; \cD^c
\quad \text{and}\quad \psi>0\, \;\text{in}\; \cD.
$$
If $\lambda\geq \lambda^*$, then $\psi= \kappa \varphi^*$ for some $\kappa>0$. In particular, $\lambda=\lambda^*$.
\end{proposition}

\begin{proof}
Fix any point $\hat{x}\in\cD$, and as before let $\breve\uptau_r$ denote the first hitting time of $\sB_r(\hat{x})$.
By definition, for $x\in\cD$ we have
$$
\psi(x)  \geq \ex^x\left[e^{\int_0^{t} (\lambda-V(X_s))\, \D{s}} \psi(X_{t})\Ind_{\{t<\uptau_\cD\}}\right],
$$
and thus by a super-martingale argument, as used in Theorem~\ref{T1.2}, we obtain for $x\in\bar\sB^c_r(\hat{x})$ that
$$
\psi(x)  \geq \ex^x\left[e^{\int_0^{t\wedge \Breve\uptau_r} (\lambda-V(X_s))\, \D{s}}
\psi(X_{t\wedge \Breve\uptau_r})\Ind_{\{t\wedge \Breve\uptau_r<\uptau_\cD\}}\right].
$$
Taking $t\to\infty$ and applying Fatou's lemma, we get
\begin{equation}\label{EP1.1A}
\psi(x)  \geq \ex^x\left[e^{\int_0^{\Breve\uptau_r} (\lambda-V(X_s))\, \D{s}}
\psi(X_{\Breve\uptau_r})\Ind_{\{\Breve\uptau_r<\uptau_\cD\}}\right], \quad x\in\bar\sB^c_r(\hat{x})\cap\cD\,.
\end{equation}
By letting $r\to 0$ in \eqref{EP1.1A} and making use of Theorem~\ref{T1.1} we have
$$
\psi(x)\geq \frac{\psi(\hat{x})}{\varphi^*(\hat{x})}\varphi^*(x), \quad x\in\cD.
$$
Clearly, this implies that if we choose $\kappa= \frac{\psi(\hat{x})}{\varphi^*(\hat{x})}$, then $\psi-
\kappa\varphi^*\geq 0$ in $\Rd$ and $\psi(\hat{x})=\kappa \varphi^*(\hat{x})$. Suppose that there exists
$x_0\in\cD$ such that for some $r>0$ we have $\psi(z)-\kappa\varphi^*(z)>0$ for $z\in\overline{\sB_r(x_0)}$.
We may choose $r$ small enough such that $\hat{x}\notin \overline {\sB_r(x_0)}\subset\cD$. Also, note that
\eqref{EP1.1A} stays valid if we change the reference point to $x_0$. Thus applying Theorem~\ref{T1.1} again,
we get
$$
0=\psi(\hat{x})-\kappa\varphi(\hat{x})\geq \ex^{\hat{x}}\left[e^{\int_0^{\Breve\uptau_r}(\lambda^*-V(X_s))\,\D{s}}
(\psi(X_{\Breve\uptau_r})-\kappa\varphi^*(X_{\Breve\uptau_r}))\Ind_{\{\Breve\uptau_r<\uptau_\cD\}}\right]\geq 0.
$$
Since $\Prob^{\hat{x}}(\Breve\uptau_r<\uptau_\cD)>0$ by \eqref{ET1.1A}, the above expression yields a
contradiction and thus no such $x_0$ exists. This proves $\psi=\kappa\varphi^*$.
\end{proof}

Now we propose a weak version of an anti-maximum principle. The difficulty in obtaining a full anti-maximum
principle is due to the lack of Hopf's lemma for a general class of operators.
Below we provide a technique which can be applied to a much larger class of operators than before.

\begin{theorem}[\textbf{Weak anti-maximum principle}]
\label{T1.4}
Suppose that the assertion of Theorem~\ref{T1.2} holds. Let $f\in \cC(\bar\cD)$ and $f\gneq 0$. Let $\cK \Subset
\cD$ be compact. Then there exists $\delta>0$ such that for every $\lambda^*<\lambda< \lambda^*+\delta$, any weak
solution of
$$
H^{\cD, V} \psi -\lambda \psi =f,
$$
satisfies $\psi<0$ in $\cK$.
\end{theorem}

\begin{proof}
We proceed by contradiction and start by assuming that no such $\delta$ exists. Hence there exist a sequence $\lambda_n
\searrow \lambda^*$ and corresponding weak solutions $\psi_n$, non-negative at a suitable point in $\cK$. By the definition
of a weak solution, for $x\in\cD$
\begin{equation}\label{ET1.4A}
\psi_n(x)= \ex^x\left[e^{\int_0^t (\lambda_n-V(X_{s}))\, \D{s}} \psi_n(X_t)\Ind_{\{t<\uptau_\cD\}}\right] +
\ex^x\left[\int_0^{t\wedge\uptau_\cD} e^{\int_0^s(\lambda_n-V(X_{r})) \, \D{r}} f(X_s)\right]
\end{equation}
holds. Note that $\liminf_{n\to\infty} \norm{\psi_n}_\infty>0$ or else, by taking the limit in \eqref{ET1.4A}, one
would obtain
$$
\ex^x\left[\int_0^{t\wedge\uptau_\cD}e^{\int_0^s(\lambda-V(X_{r}))\,\D{r}} f(X_s)\right]=0\,, \quad x\in\cD,\quad t>0,
$$
which is impossible as $f\gneq 0$. Now we split the proof in two cases.

\medskip
\noindent
{\bf Case 1:}\, First suppose $\limsup_{n\to\infty} \norm{\psi_n}_\infty<\infty$. Then we can extract a weakly
convergent subsequence, which we keep denoting in the same way, such that $\psi_n\rightharpoonup \psi_0 \in
L^2(\cD)$. Recall that for
\begin{equation}\label{ET4.4B}
T^{\cD, V-\lambda}(t, x, y) =
\ex^0_{\Prob_S}\left[p_{S^\Psi_t} (x-y) \ex^{x, y}_{0, S^\Psi_t} \left[e^{\int_0^t (\lambda-V(B_s))\, \D{s}}
\Ind_{\{\uptau_\cD>t\}}\right]\right],
\end{equation}
we have
$$
\ex^x\left[e^{\int_0^t (\lambda-V(X_{s}))\, \D{s}} \psi_n(X_t)\Ind_{\{t<\uptau_\cD\}}\right]=
\int_{\cD} T^{\cD, V-\lambda_n}(t, x, y)\psi_n(y)\,\D{y}.
$$
Since for every fixed $x\in\cD$ we have from \eqref{ET4.4B}
$$
T^{\cD, V-\lambda_n}(t, x, \cdot)\to T^{\cD, V-\lambda^*}(t, x, \cdot) \quad \text{in}\; L^2(\cD),
$$
it follows from \eqref{ET1.4A} that $\psi_n(x)\to\psi_0(x)$ as $n\to\infty$. From the assumption that
$\limsup_{n\to\infty} \norm{\psi_n}_\infty<\infty$ this furthermore gives $\norm{\psi_n-\psi_0}_{2, \cD}\to 0$,
and thus $\norm{\psi_n-\psi_0}_\infty\to 0$ as $n\to\infty$, using again \eqref{ET1.4A}. For fixed $t$, denote
$$
g(x)=\ex^x\left[\int_0^{t\wedge\uptau_\cD} e^{\int_0^s(\lambda^*-V(X_{r})) \, \D{r}}
f(X_s)\right]\in L^\infty(\cD).
$$
Taking the limit in \eqref{ET1.4A} we obtain
$$
\psi_0 (x) = T^{\cD, V-\lambda^*}_t \psi_0(x) + g(x), \quad x\in\cD\,.
$$
Note that $\psi_0\neq 0$ since $g\neq 0$. Hence we have a non-trivial solution for $(I-T^{\cD, V-\lambda^*}_t)\xi=g$
where $T^{\cD, V-\lambda^*}_t$ is a compact, self-adjoint operator. Thus by Fredholm alternative $g\in \mathrm{Ker}
(I-T^{\cD, V-\lambda^*}_t)^\perp$, implying $\langle g, \varphi^*\rangle =0$. However, this is not possible as
$g\gneq 0$ and $\varphi^*$ is also positive in $\cD$, which is a contradiction.

\medskip
\noindent
{\bf Case 2:}\, Next suppose $\limsup_{n\to\infty} \norm{\psi_n}_\infty=\infty$. Define $\tilde\psi_n =
\frac{1}{\norm{\psi_n}_\infty}\psi_n$. Repeating the argument of the previous case, we find $\psi_0\in
\cC(\bar\cD)$ satisfying
\begin{equation}
\psi_0 (x) = T^{\cD, V-\lambda^*}_t \psi_0,
\end{equation}
and $\norm{\tilde\psi_n-\psi_0}_\infty\to 0$ as $n\to\infty$. By the uniqueness of the principal eigenfunction we
have $\psi_0=\kappa\varphi^*$, for $\kappa\neq 0$. Note that $\kappa< 0$ is not possible, as this would imply for
large enough $n$ that $\psi_n<0$ on $\cK$, contradicting the assumption. In case that $\kappa>0$ we infer that
$\psi_n$ is strictly positive on $\cK$ for all large enough $n$. From \eqref{ET1.4A} we have
\begin{equation}\label{ET1.4B}
\psi_n(x)\geq  \ex^x\left[e^{\int_0^t (\lambda_n-V(X_{s}))\, \D{s}} \psi_n(X_t)\Ind_{\{t<\uptau_\cD\}}\right] .
\end{equation}
Choose a point $\hat{x}\in\cK$ and consider the potential
\[
\widetilde{V}=\left\{\begin{array}{ll}
V & \text{for}\; x\in \sB_r(\hat{x}) \cap \cD,
\\
\norm{V}_\infty +1 & \text{for}\; x\in \sB_r(\hat{x}).
\end{array}
\right.
\]
We can choose $r$ small enough such that $\psi_n$ is positive in $\sB_r$ for all sufficiently large $n$. By the
same argument as used in Theorem~\ref{T1.1}, we can find $\varepsilon>0$ such that $\lambda^*_{\varepsilon}>\lambda^*$,
where $\lambda^*_{\varepsilon}$ is the principal eigenvalue of $H^{\cD, \tilde V}$. Now choose $n$ large
such that $\lambda^*_{\varepsilon}>\lambda_n$. Then following the argument of Theorem~\ref{T1.1} it is seen that
$$
\ex^x\left[e^{\int_0^{t} (\lambda_n-V(X_s))\, \D{s}} \psi_n(X_{t})\Ind_{\{t<\uptau_\cD\}}
\Ind_{\{t\leq \Breve\uptau_r\}}\right]\to 0,
$$
as $t\to\infty$. Thus we obtain from \eqref{ET1.4B} (see also \eqref{EP1.1A}) that for $x\in \sB^c_r(\hat{x})$,
$$
\psi_n(x)\geq  \ex^x\left[e^{\int_0^{\breve\uptau_r} (\lambda_n-V(X_{s}))\, \D{s}}
\psi_n(X_{\breve\uptau_r})\Ind_{\{\breve\uptau_r<\uptau_\cD\}}\right]>0\,.
$$
Hence for all large enough $n$ we have $\psi_n>0$ in $\cD$ and it is a weak super-solution. Thus by Proposition~\ref{P1.1}
$\psi_n=\kappa_n\varphi^*$ for some $\kappa_n>0$ and $\lambda^*=\lambda_n$. This also implies $f=0$, which is a
contradiction. Hence no $\lambda_n$ with such property exists. This proves the existence of $\delta$ as stated in the
assertion.
\end{proof}

The above result can be upgraded to a full anti-maximum principle by restricting to fractional Schr\"{o}dinger operators,
for which a counterpart of Hopf's lemma is available. This is the content of the following result.

\begin{theorem}[\textbf{Anti-maximum principle for fractional Schr\"odinger operators}]
\label{antimaxx}
Let $\cD$ be a $\cC^{1,1}$ domain and $\alpha\in(0,2)$. Consider $f\in \cC(\bar\cD)$ and $f\gneq 0$. Furthermore,
assume that $V$ is H\"{o}lder continuous on $\cD$. Then there exists $\delta>0$ such that for every $\lambda^*<\lambda<
\lambda^*+\delta$, the weak solution of
$$
(-\Delta)^{\nicefrac{\alpha}{2}}\psi + V\psi - \lambda\psi =f\;\;\text{in}\;\; \cD,\quad \psi=0 \;\; \text{in}\;\;\cD^c,
$$
satisfies $\psi<0$ in $\cD$.
\end{theorem}

\begin{proof}
We proceed along the argument in Theorem~\ref{T1.4}. It is straightforward to see that the argument in Case 1 applies
in a similar way. For Case 2 we only need to consider the situation where $\psi_0<0$ in $\cD$. Recall that $\psi_0=
\kappa\varphi^*$ for some $\kappa<0$. Since $V$ and $f$ are continuous, we see that $\tilde\psi_n$ is a viscosity solution
(see Remark \ref{R3.1} above) to
$$
(-\Delta)^{\nicefrac{\alpha}{2}} \tilde\psi_n  + V\tilde\psi_n -\lambda_n \tilde\psi_n =\tilde{f}_n\; \text{in}\; \cD,
\quad \psi=0 \;\; \text{in}\; \cD^c,
$$
where $\tilde{f}_n=\frac{1}{\norm{\psi_n}} f$. Let $\delta(\cdot)$ be the distance function from the boundary of $\cD$.
Then we obtain from \cite[Theorem~1.2]{RS-14} that for a positive $\beta<\min\{\alpha, 1-\alpha\}$ the function
$\frac{\tilde\psi_n} {\delta^{\nicefrac{\alpha}{2}}}$ is in $\cC^\beta(\bar\cD)$, uniformly in $n$, and thus
\begin{equation}\label{101}
\sup_{x\in\bar\cD}\left|\frac{\tilde\psi_n(x)}{\delta^{\nicefrac{\alpha}{2}}(x)} -
\frac{\psi_0(x)}{\delta^{\nicefrac{\alpha}{2}}(x)}\right|\to 0,
\end{equation}
as $n\to \infty$, along a suitable subsequence. Since $\tilde\psi_n$ is non-negative at some point in $\cD$, we find a
sequence of points $\seq x\subset\cD$ such that $\tilde\psi_n(x_n)\geq 0$. Therefore, passing to the limit and assuming
$x_n\to x_0$, we obtain from \eqref{101}
\begin{equation}\label{102}
\frac{\psi_0(x_0)}{\delta^{\nicefrac{\alpha}{2}}(x_0)}\geq 0.
\end{equation}
Since $\psi_0<0$ in $\cD$, it is clear that $x_0\in\partial\cD$, hence $\psi_0$ attains its maximum at $x_0$. Since $V$
is H\"{o}lder continuous, the equation
$$
(-\Delta)^{\nicefrac{\alpha}{2}} \psi_0  + V \psi_0 =\lambda^* \psi_0 \;\; \text{in}\;\; \cD, \;\;\; \psi=0
\;\; \text{in}\; \cD^c,
$$
holds pointwise, see \cite{RS-14}. Then
$$
(-\Delta)^{\nicefrac{\alpha}{2}} \psi_0  + (V-\lambda^*)^+ \psi_0\leq 0, \quad \text{pointwise in}\; \cD.
$$
Thus by \cite[Lem.~1.2]{GS-16} we obtain
$$
\lim_{\cD\ni y\to x_0}\frac{\psi_0(y)}{\delta^{\nicefrac{\alpha}{2}}(y)}<0.
$$
This contradicts \eqref{102}, and the remaining part of the proof can be completed as in Theorem~\ref{T1.4}.
\end{proof}

The Feynman-Kac representation is also useful in obtaining a maximum principle for narrow domains for $\Psidel$.
This gives a counterpart to non-local operators of the known result for elliptic operators. A version of this
result has been established for classical solution of the fractional Laplacian \cite{FW14}. We are not aware of
any such results for this general class of operators.
\begin{theorem}[\textbf{Maximum principle in narrow domains}]
\label{narrow}
Suppose that $\cD$ is a convex (bounded or unbounded) domain of finite inradius, and let $\varphi$ be a weak
sub-solution of
$$
H^{\cD, V} \varphi \leq 0 \;\text{in}\; \cD, \quad \varphi\leq 0\;\; \text{in}\;\; \cD^c.
$$
There exists a constant $\theta > 0$, independent of $\Psi$ and  $\cD$, such that if
$$
\norm{V^-}_{\infty, \cD} -\inf_{\cD}V^+\, < \theta\, \Psi([\inrad \cD]^{-2}),
$$
then $\varphi \leq 0$ in $\cD$. Moreover, either $\varphi=0$ or $\varphi<0$ in $\cD$.
\end{theorem}

\begin{proof}
By the definition of a sub-solution, we have for all $t\geq 0$ and $x\in\cD$ that
\begin{align*}
\varphi(x) &\leq \ex^x\left[ e^{-\int_0^{t\wedge \uptau_\cD} V(X_s)\, \D{s}} \varphi(X_{t\wedge\uptau_\cD})\right] \\
& \leq \ex^x\left[ e^{-\int_0^{t} V(X_s)\, \D{s}} \varphi(X_{t})\Ind_{\{t<\uptau_\cD\}}\right]
\leq \ex^x\left[ e^{-\int_0^{t} V(X_s)\, \D{s}} \varphi^+(X_{t})\Ind_{\{t<\uptau_\cD\}}\right].
\end{align*}
This representation was used as a key tool in \cite{BL}. Let $\varphi^+\neq 0$ and $x^*$ be a global maximizer of
$\varphi^+$. Write $r=\dist(x^*, \partial\cD)$. Then by \cite[Th.~3.2]{BL} there exists a universal constant
$\theta \approx 0.083$ satisfying
$$
\norm{V^-}_{\infty, \cD} -\inf_{\cD}V^+\, \geq \theta\, \Psi(r^{-2})\geq \theta\, \Psi([\inrad \cD]^{-2}).
$$
This leads to a contradiction to the assumption with the above choice of $\theta$. Hence $\varphi^+=0$.

To prove the second claim, assume that $\varphi(\hat{x})<0$, for some $\hat{x}\in\cD$. Choose $r$ small enough
such that $\sB_r(\hat x)\subset\cD$. Note that for $\psi=-\varphi$ we have
$$
\psi(x)\geq \ex^x\left[ e^{-\int_0^{t} V(X_s)\, \D{s}} \psi(X_{t})\Ind_{\{t<\uptau_\cD\}}\right].
$$
By a super-martingale argument, as used in Theorem~\ref{T1.2}, we obtain
$$
\psi(x)\geq \ex^x\left[ e^{-\int_0^{t\wedge\breve\uptau_r} V(X_s)\, \D{s}}
\psi(X_{t\wedge\breve\uptau_r})\Ind_{\{t\wedge\breve\uptau_r<\uptau_\cD\}}\right].
$$
Letting $t\to\infty$ and applying Fatou's lemma, we obtain for all $x\in \sB^c_r(\hat x)$
$$
\psi(x) \geq \ex^x\left[ e^{-\int_0^{\breve\uptau_r} V(X_s)\, \D{s}}
\psi(X_{\breve\uptau_r})\Ind_{\{\breve\uptau_r<\uptau_\cD\}}\right] >0\,.
$$
This proves the result.
\end{proof}

In the remaining part of this subsection we establish a refined version of the elliptic ABP estimate with the
help of Theorem~\ref{T1.1}. We begin with the following result, which might be known in some form but we provide
here a proof for self-containedness.
\begin{lemma}\label{L4.3}
Let $V$ be bounded. Then the map $\Lambda(s)=\lambda^*_{sV}$, where $\lambda^*_{sV}$ is the principal eigenvalue
with potential $sV$, is concave in $s$. Moreover, $\Lambda$ is Lipschitz continuous with constant $\norm{V}_{\infty, \cD}$.
\end{lemma}

\begin{proof}
Recall the semigroup operator $T^{\cD, V}_t $ and expression \eqref{EL1.2A}.
Note that for $s_\theta=\theta s_1 + (1-\theta) s_2$, $\theta\in[0, 1]$, we get by the H\"{o}lder inequality
$$
T^{\cD, s_\theta V}(t, x, y)\leq \left(T^{\cD, s_1 V}(t, x, y)\right)^\theta \,
\left(T^{\cD, s_2 V}(t, x, y)\right)^{1-\theta}.
$$
Thus
\begin{align*}
e^{-\Lambda(s_\theta) t}
& =
\sup \left\{ \int_{\cD}\int_{\cD} \psi(x) T^{\cD, s_\theta V}(t, x, y) \psi(y) \, \D{x}\D{y}
\; : \; \norm{\psi}_{2, \cD}=1,\,\psi\geq 0\right\}
\\
&\leq
\sup \Bigl\{ \Bigl(\int_{\cD}\int_{\cD} \psi(x) T^{\cD, s_1 V}(t, x, y) \psi(y) \, \D{x}\D{y}\Bigr)^\theta
\Bigl(\int_{\cD}\int_{\cD} \psi(x) T^{\cD, s_2 V}(t, x, y) \psi(y) \, \D{x}\D{y}\Bigr)^{1-\theta}\;
\\
&\, \hspace{1cm}  : \; \norm{\psi}_{2, \cD}=1,\,\psi\geq 0\Bigr\}
\\
&\leq e^{-\theta\Lambda(s_1) t} e^{-(1-\theta)\Lambda(s_2) t}.
\end{align*}
Hence we have $\Lambda(s_\theta)\geq \theta\Lambda(s_1) + (1-\theta)\Lambda(s_2)$. This proves concavity of
$\Lambda$. Since for $f\geq 0$,
$$
T^{\cD, V_1}_t f \leq e^{t \norm{V_1-V_2}_\infty} T^{\cD, V_2}_t f,
$$
the Lipschitz continuity of $\Lambda$ is straightforward.
\end{proof}

Now we are ready to prove a refined elliptic ABP-type estimate
\begin{theorem}[\textbf{Refined elliptic ABP estimate}]
\label{T4.7}
Let $\Psi \in \cB_0$ be strictly increasing and satisfy Assumptions~\ref{WLSC}(i) and ~\ref{As4.1}. Suppose that
$\lambda^*>0$. Let $\varrho^* =\inf\{s\geq 1:  \Lambda(s)\leq 0\}$ and $1<\varrho<\varrho^*$. Then for every
bounded weak super-solution $\varphi$ of
$$
H^{\cD, V} \varphi \leq f \;\; \mbox{in}\,\; \cD,\quad \varphi=0 \;\; \text{in}\;\; \cD^c,
$$
with $f\in L^{p^*}(\cD)$, $p>\frac{d}{2{\underline\mu}}$ and $p^*=\frac{\varrho\, p}{\varrho-1}$, there exists a constant
$C = C(p, k, d, \Psi, \cD, V)$ such that
\begin{equation}\label{ET1.6A}
\sup_{\cD} \varphi^+ \leq C\, \norm{f}_{p^*, \cD}\,.
\end{equation}
\end{theorem}

\begin{proof}
From the continuity of $\Lambda$, see Lemma~\ref{L4.3}, we have $\varrho^*>1$. Without loss of generality we
assume that $f\geq 0$, supported on $\bar\cD$. We note from the proof Theorem~\ref{T1.6} that
\begin{align}\label{ET4.7A}
\varphi(x) &\leq \ex^x\left[ e^{-\int_0^{t\wedge\uptau_\cD} V(X_s)\, \D{s}} \varphi^+(X_{t\wedge\uptau_\cD})\right]+
\ex^x\left[\int_0^{\uptau_\cD} e^{-\int_0^t V(X_s) \D{s}}f(X_s)\, \D{t}\right],\nonumber
\\
&\leq \ex^x\left[ e^{-\int_0^{t\wedge\uptau_\cD} V(X_s)\, \D{s}} \varphi^+(X_{t\wedge\uptau_\cD})\right]+
\ex^x\left[\int_0^{\uptau_\cD} e^{-\int_0^t \varrho V(X_s) \D{s}} \D{t}\right]^{\frac{1}{\varrho}}
\ex^x\left[\int_0^{\uptau_\cD} f^{\frac{\varrho}{\varrho-1}}(X_t) \D{t}\right]^{\frac{\varrho-1}{\varrho}}.
\end{align}
We estimate the term at the right hand side of \eqref{ET4.7A}. From the proof of Theorem~\ref{T1.6} we have
\begin{equation}\label{ET4.7B}
\ex^x\left[\int_0^{\uptau_\cD} f^{\frac{\varrho}{\varrho-1}}(X_t) \D{t}\right]\leq
C_1 \norm{f}_{p^*, \cD}^{\frac{\varrho}{\varrho-1}},
\end{equation}
where the constant $C_1$ depends on $\Psi, \cD, p, \varrho$. By definition we have $\Lambda(\varrho)>0$. Thus
by Lemma~\ref{L1.2}(i) we find $n$ large enough such that $\lambda^*_{\varrho,n}>0$, where $\lambda^*_{\varrho,n}$
is the principal eigenvalue in $\cD_n$ with potential $\varrho V$. We fix such an $n$ and denote the corresponding
principal eigenpair by $(\varphi^*_{\varrho,n}, \lambda^*_{\varrho,n})$. Also, note that $\varphi^*_{\varrho,n}\in
\cC(\cD_n)$ and $\varphi^*_{\varrho,n}>0$ in $\cD_n$. We show that
\begin{equation}\label{ET4.7C}
\sup_{x\in\cD} \ex^x\left[\int_0^{\uptau_\cD} e^{-\int_0^t \varrho V(X_s) \D{s}} \D{t}\right]\leq C_2,
\end{equation}
for a constant $C_2$, dependent only on $\lambda^*_{\varrho,n}, \varphi^*_{\varrho,n}$. Indeed, we have
\begin{align*}
\ex^x\left[\int_0^{\uptau_\cD} e^{-\int_0^t \varrho V(X_s) \D{s}} \D{t}\right] &=
\ex^x\left[\int_0^{\infty} e^{-\int_0^t \varrho V(X_s) \D{s}} \Ind_{\{t<\uptau_\cD\}}\D{t}\right]
\\
& \leq \frac{1}{\min_{\cD}\varphi^*_{\varrho,n}} \ex^x\left[\int_0^{\infty} e^{-\int_0^t \varrho V(X_s) \D{s}}
\varphi^*_{\varrho,n}(X_t) \Ind_{\{t<\uptau_\cD\}}\D{t}\right]
\\
& \leq \frac{1}{\min_{\cD}\varphi^*_{\varrho,n}} \int_0^{\infty} \ex^x\left[e^{-\int_0^t \varrho V(X_s) \D{s}}
\varphi^*_{\varrho,n}(X_t) \Ind_{\{t<\uptau_{\cD_n}\}}\right] \D{t}
\\
&= \frac{1}{\min_{\cD}\varphi^*_{\varrho,n}} \int_0^{\infty} e^{-\lambda^*_{\varrho,n} t}\varphi^*_{\varrho,n}(x) \D{t}
\\
&\leq \frac{\max_{\cD}\varphi^*_{\varrho,n}}{\min_{\cD}\varphi^*_{\varrho,n}}\, \frac{1}{\lambda^*_{\varrho,n}}.
\end{align*}
This proves \eqref{ET4.7C}. Now we estimate the first term at the right hand side of \eqref{ET4.7A}. Recall that
$(\varphi^*_n, \lambda^*_n)$ is the principal eigenpair in $\cD_n$ with potential $V$. Using Lemma~\ref{L1.2}(i)
we choose $n$ large enough so that $\lambda^*_n>0$. Then
\begin{align}\label{ET4.7D}
\lim_{t\to\infty} \ex^x\left[ e^{-\int_0^{t} V(X_s)\, \D{s}} \varphi^+(t)\Ind_{\{t<\uptau_\cD\}}\right]
&=
\frac{\sup_{\Rd} \varphi^+}{\min_{\cD}\varphi^*_n} \lim_{t\to\infty} \ex^x\left[ e^{-\int_0^{t} V(X_s)\, \D{s}}
\varphi^*_n(X_t)\Ind_{\{t<\uptau_\cD\}}\right]\nonumber
\\
&\leq \frac{\sup_{\Rd} \varphi^+}{\min_{\cD}\varphi^*_n} \lim_{t\to\infty} e^{-\lambda^*_n t}
\ex^x\left[ e^{\int_0^{t}(\lambda^*_n-V(X_s))\, \D{s}} \varphi^*_n(X_t)\Ind_{\{t<\uptau_{\cD_n}\}}\right]\nonumber
\\
&= \frac{\sup_{\Rd} \varphi^+}{\min_{\cD}\varphi^*_n} \lim_{t\to\infty} e^{-\lambda^*_n t}\varphi^*_n(x)=0\,.
\end{align}
Thus the claim follows by a combination of \eqref{ET4.7B}-\eqref{ET4.7D}.
\end{proof}

\section{\textbf{Liouville-type theorems}}

\medskip
\subsection{$\Psi$-harmonic functions and a Liouville theorem}

To conclude, we prove several Liouville-type results for a class of non-local Schr\"odinger operators.
Denote, as above, by $\uptau_\cD$ the first exit time of $\pro X$ from $\cD$. We say that a function $\varphi$ is
\emph{$\Psi$-harmonic} if for every bounded domain $\cD$ and $t>0$
$$
\varphi(x)=\ex^x[\varphi(X_{t\wedge\uptau_\cD})], \quad x\in\cD,
$$
holds. Note that whenever $\varphi$ is bounded, by choosing a sequence of compact sets $\seq\cD$ increasing to $\Rd$ we
get
$$
\varphi(x)=\ex^x[\varphi(X_t)], \quad x\in\Rd, \; t\geq 0.
$$
\begin{theorem}[{\bf Liouville-type theorem}]\label{T4.8}
Suppose that $\Psi \in \cB_0$ is positive and satisfies Assumption~\ref{WLSC}(i) with $\underline{\theta}=0$. Then there
are no bounded $\Psi$-harmonic functions other than constants.
\end{theorem}

\begin{proof}
The argument below took some inspiration from \cite{RS16}.
First notice that \eqref{HW} is satisfied. Thus for every $t>0$ the transition density $q_t$ has a bounded derivative
\cite[Lem.~3.1]{KS13}. Fix $\rho\geq1$ and consider the process $\pro Y$ with $Y_t=\frac{1}{\rho} X_t$, where $\pro X$
is subordinate Brownian motion starting from $0$, as above. Due to the scaling property of Brownian motion, we observe
that
$$
Y_\cdot = \frac{1}{\rho} B_{S^\Psi_\cdot} \Dst  B_{\rho^{-2}S^\Psi_\cdot}\Dst B_{S^{\Psi_\rho}_\cdot},
$$
where $\pro {S^{\Psi_\rho}}$ is a subordinator with Bernstein function $\Psi_\rho(u)=\Psi(u \rho^{-2})$. Note that
$\Psi_\rho$ satisfies \eqref{HW} and thus $\pro Y$ has a smooth transition density function $q^\rho_t$.
Fix $t=t_\rho=\frac{1}{\Psi_\rho(1)}$. Our proof below crucially relies on the following two claims:
\begin{align}
\sup_{\rho\geq 1} \, \sup_{y\in\Rd} \abs{\grad q^\rho_{t}(y)} &<\; \infty,\label{ET4.8A}
\\
\sup_{\rho\geq 1} \int_{\Rd} q^\rho_t (y) (1+\abs{y})^\delta &<\; \infty, \quad \delta\in(0, 2\underline{\mu}).
\label{ET4.8B}
\end{align}

Now we prove these claims. Recall that $\pro Y$ has L\'{e}vy exponent $\Phi_\rho(u)=\Psi_\rho(u^2)$.
Therefore, by inverse Fourier transform we get that
$$
q^\rho_t(y)=\frac{1}{(2\pi)^d}\int_{\Rd} e^{-i y\cdot \xi} e^{-t\Psi_\rho(\abs{\xi}^2)}\D{\xi}=
\frac{1}{(2\pi)^d}\int_{\Rd} \cos (y\cdot\xi) e^{-t\Psi(\rho^{-2}\abs{\xi}^2)} \D{\xi}.
$$
By the assertion, $\Psi(\abs{\xi}^2\rho^{-2})\geq \underline{c}\, \abs{\xi}^{2\underline{\mu}}\Psi(\rho^{-2})$ for
$\abs{\xi}\geq 1$. Thus by differentiation in the above expression we obtain
\begin{align*}
\abs{\grad q_t(y)} &\leq \frac{d}{(2\pi)^d}\int_{\Rd} \abs{\xi} e^{-t\Psi(\rho^{-2}\abs{\xi}^2)} \D{\xi}
\\
& \leq \frac{d}{(2\pi)^d}\left(\int_{\abs{\xi}\leq 1} \abs{\xi} e^{-t\Psi(\rho^{-2}\abs{\xi}^2)} \D{\xi} +
\int_{\abs{\xi}> 1} \abs{\xi} e^{-\underline{c}\abs{\xi}^{2\underline{\mu}}} \D{\xi} \right)
\\
& \leq \frac{d}{(2\pi)^d}\left(\abs{\sB_1(0)}
+ \int_{\abs{\xi}> 1} \abs{\xi} e^{-\underline{c}\abs{\xi}^{2\underline{\mu}}} \D{\xi} \right).
\end{align*}
This proves \eqref{ET4.8A}. Next we show \eqref{ET4.8B}. Notice that the above estimate also gives
\begin{equation}\label{ET4.8C}
q^\rho_t (y)\leq \frac{d}{(2\pi)^d}\left(\abs{\sB_1(0)}+
\int_{\abs{\xi}> 1}e^{-\underline{c}\abs{\xi}^{2\underline{\mu}}}\D{\xi}\right), \quad y\in\Rd, \; \rho\geq 1\,.
\end{equation}
Denote by $C_1$ the right hand side of \eqref{ET4.8C}. By \cite[Cor.~7]{BGR14b} there exists a
constant $C=C(d)$, dependent on $d$ alone, satisfying
\begin{equation}\label{ET4.8D}
q^\rho_t(y) \leq  \frac{C t}{\abs{y}^d} \Psi_\rho\left(\frac{1}{\abs{y}^2}\right)=  \frac{C t}{\abs{y}^d}
\Psi(\abs{\rho y}^{-2}), \quad \abs{y}>0.
\end{equation}
By the WLSC property (Assumption~\ref{WLSC}(i)) we get that for $\abs{y}\geq 1$
$$
\Psi(\abs{\rho y}^{-2}) \abs{y}^{2\underline{\mu}} \underline{c}\leq \Psi(\rho^{-2} )= \frac{1}{t},
$$
and thus by \eqref{ET4.8D} we have
\begin{equation}
\label{qrho}
q^\rho_t (y) \leq C \frac{1}{\underline{c}}\,  \abs{y}^{-d-2\underline{\mu}},\, \quad \abs{y}\geq 1.
\end{equation}
Let $0 < \delta < 2\underline{\mu}$. Then using \eqref{ET4.8C}, \eqref{ET4.8D} and \eqref{qrho} we obtain that for
every $\rho\geq 1$ and $t=t_\rho$,
\begin{align*}
\int_{\Rd} q^\rho_t (y) (1+\abs{y})^\delta\, \D{y}
& \leq
\int_{\abs{y}\leq 1} q^\rho_t (y) (1+\abs{y})^\delta\, \D{y}+ \int_{\abs{y}>1} q^\rho_t (y)
(1+\abs{y})^\delta\, \D{y}
\\
& \leq
C_1 2^\delta \abs{\sB_1(0)} + \frac{C}{\underline{c}}\int_{\abs{y}>1} \abs{y}^{-d-2\underline{\mu}}
(1+\abs{y})^\delta\, \D{y}
\\
&\leq
C_1 2^\delta \abs{\sB_1(0)} + \frac{C 2^\delta}{\underline{c}}\int_{\abs{y}>1}
\abs{y}^{-d-2\underline{\mu}+\delta}\, \D{y}
\\
&=
C_1 2^\delta \abs{\sB_1(0)} + \frac{C 2^\delta}{\underline{c}}\frac{d\,\abs{\sB_1(0)}}{2\underline{\mu}-\delta}\,.
\end{align*}
This proves \eqref{ET4.8B}.

Now we are ready to complete the proof of the theorem. Let $\varphi$ be a bounded $\Psi$-harmonic function,
fix $\rho\geq 1$, and define $v(x)=\varphi(x\rho)$. Then
$$
v(x) = \varphi(\rho x) =\ex^{\rho x}[v(X_t)]=\ex^0[\varphi(X_t + \rho x)]= \ex^0[v(Y_t + x)] =
\int_{\Rd} v(y+x) q^\rho_t (y)\, \D{y},
$$
for all $t>0$. Choosing, in particular, $t=t_\rho=\frac{1}{\Psi(\rho^{-2})}$, we have for every $\abs{x}\leq 1$
and $\kappa>0$ that
\begin{align*}
\abs{v(x)-v(0)}& =\left|\int_{\Rd} v(y+x) q^\rho_t (y)\, \D{y}-\int_{\Rd} v(y) q^\rho_t (y)\, \D{y}\right|
\\
&=
\left|\int_{\Rd} v(y) q^\rho_t (y-x)\, \D{y}-\int_{\Rd} v(y) q^\rho_t (y)\, \D{y}\right|
\\
&\leq
\int_{\abs{y}\leq \kappa } \abs{v(y)} \abs{q^\rho_t (y-x)-  q^\rho_t (y)}\, \D{y} + \int_{\abs{y}> \kappa}
\abs{v(y)} \abs{q^\rho_t (y-x)-  q^\rho_t (y)}\, \D{y}
\\
&\leq
\left(\sup_{x\in\Rd} \abs{\grad q^\rho_t(x)}\right) \abs{x} \norm{v}_\infty \kappa^d \abs{\sB_1(0)} +
\norm{v}_\infty \kappa^{-\delta}\int_{\abs{y}>\kappa} q^\rho_t (y) (1+\abs{y})^\delta\, \D{y}
\\
&\,
\qquad +\norm{v}_\infty \kappa^{-\delta}\int_{\abs{y}>\kappa} q^\rho_t (y-x) (1+\abs{y})^\delta\, \D{y}
\\
&\leq
C_3 \norm{v}_\infty \left( \abs{x} \kappa^d + \kappa^{-\delta}\right) + \norm{v}_\infty \kappa^{-\delta}
\left(\sup_{\abs{x}\leq 1}\, \sup_{y\in\Rd} \frac{(1+\abs{y+x})^\delta}{(1+\abs{y})^\delta}\right)
\int_{\Rd} q^\rho_t (y) (1+\abs{y})^\delta\, \D{y}
\\
&\leq
C_4 \norm{v}_\infty \left( \abs{x} \kappa^d + \kappa^{-\delta}\right),
\end{align*}
with some constants $C_3, C_4$, where in the fourth and sixth lines we used \eqref{ET4.8A} and \eqref{ET4.8B}.
Let $x\neq 0$ and fix $\kappa=\abs{x}^{-\frac{1}{d+\delta}}$. Applying this in the above, we find
$$
\abs{v(x)-v(0)}\leq 2 C_4 \norm{v}_\infty \abs{x}^{\frac{\delta}{d+\delta}},
$$
which implies
\begin{equation}\label{ET4.8E}
\abs{\varphi(\rho x)-\varphi(0)}\leq 2 C_4\norm{v}_\infty \abs{x}^{\frac{\delta}{d+\delta}}, \quad \rho\geq 1,
\; \abs{x}\leq 1.
\end{equation}
In \eqref{ET4.8E}, by replacing $x$ by $\rho^{-1} x$, for $\rho>\abs{x}$ we obtain
$$
\abs{\varphi(x)-\varphi(0)}\leq 2 C_4\, \norm{v}_\infty
\abs{x}^{\frac{\delta}{d+\delta}}\rho^{-\frac{\delta}{d+\delta}},
$$
and then letting $\rho\to\infty$, we find $\varphi(x)=\varphi(0)$ for all $x\in\Rd$.
\end{proof}

\subsection{Liouville-type theorem for semi-linear equations}

In this section we obtain non-existence results for bounded positive super-solutions of
$$
\Psidel\varphi\geq \varphi^p \quad \text{in}\;\; \Rd,
$$
for suitable choices of $p$.

First we recall some standard facts from the potential theory of non-local operators and jump processes,
for background we refer to \cite{B09}. It is well-known that the L\'{e}vy measure $\nu$ corresponding to
subordinate Brownian motion is isotropic and unimodal \cite{BGR14b}. For a domain $\cD$, the occupation
measure in $\cD$ is defined as
$$
G_{\cD}(x, A)=\ex^x\left[\int_0^{\uptau_\cD} \Ind_{A}(X_s) \, \D{s}\right]\,.
$$
The corresponding density is the Green function of the set $\cD$. Note that $G_{\cD}(x, \cD) =
\ex^x[\uptau_\cD]$ holds. Another important object related to $\pro X$ is the harmonic measure of the set
$\cD$ defined by
$$
P_\cD(x, A) =\ex^x[\Ind_{\{\uptau_\cD<\infty\}}\, \Ind_A(X_{\uptau_\cD}) ].
$$
The density kernel of the harmonic measure, whenever it exists, is the Poisson kernel. For a L\'evy
process with jumps, the relationship between the harmonic measure and the occupation measure is provided by the
Ikeda-Watanabe formula (see \cite{IW62})
\begin{equation}\label{IW}
P_{\cD}(x, A)= \int_{\cD}\int_{A}\nu(z-y)\, \D{z}\, G_{\cD}(x, \D{y}), \quad A\subset\bar\cD^c\,.
\end{equation}

We also recall that a L\'evy process $\pro Z$ is recurrent whenever $\Prob_Z\left(\liminf_{t\to\infty}
|Z_t| = 0\right) = 1$ and transient whenever $\Prob_Z\left(\lim_{t\to\infty} |Z_t| = \infty\right) = 1$.
The standard Chung-Fuchs condition gives a criterion to verify recurrence in terms of the L\'evy exponent.
It is also a well-established fact that a L\'{e}vy process is either recurrent or transient, and for
$d\geq 3$ every L\'evy process is transient. In our context, we only need the following observations.
\begin{remark}\label{R5.1}
\hspace{100cm}
\begin{enumerate}
\item[(i)]
Specifically, for a subordinate Brownian motion $\pro X$ obtained under a subordinator with Laplace exponent
$\Psi$, one equivalent expression of the standard Chung-Fuchs criterion says that the process is recurrent if
and only if
$$
\int_{0+}^r \frac{u^{\nicefrac{d}{2}-1}}{\Psi(u)}\, \D{u} = \infty
$$
for every $r>0$.
\item[(ii)]
Recall $\bar\mu$ in Assumption \ref{WLSC}(ii). We note that from this condition it follows that for $d>2\bar\mu$
the process $\pro X$ is transient. For every $u\in (0, 1)$ we have
$$
\frac{1}{\Psi(u)}\leq \frac{\bar c}{\Psi(1)} u^{-\bar\mu},
$$
hence $\int_{0}^1 \frac{u^{\nicefrac{d}{2}-1}}{\Psi(u)}<\infty$,  which implies that $\pro X$ is transient by
part (i) above.
\end{enumerate}
\end{remark}

The following result is well-known, however, we include a short proof to highlight the contrast with the
next result in our approach.
\begin{theorem}[\textbf{Liouville theorem: recurrent case}]
Let $\Psi \in \cB_0$ such that the subordinate Brownian motion $\pro X$ is recurrent. Furthermore, assume
that $\pro X$ has a strictly positive transition density for some $t>0$. If $\varphi$ is a non-negative weak
super-solution of
$$
\Psidel \varphi\geq 0 \quad \text{in} \;\; \Rd,
$$
then $\varphi$ is a constant.
\end{theorem}

\begin{proof}
Since $\varphi$ is a super-solution, we have
$$
\varphi(x)\geq \ex^x[\varphi(X_{t\wedge\uptau_{B_r(x)}})], \quad t>0, \; r>0\,.
$$
By letting $r\to\infty$ and using Fatou's lemma, we obtain
$$
\varphi(x)\geq \ex^x[\varphi(X_{t})], \quad t>0\,.
$$
Using as in Theorem~\ref{T1.2} above that $(\varphi(X_t))_{t\geq 0}$ is a super-martingale, by Doob's martingale
convergence theorem it follows that $\lim_{t\to\infty} \varphi(X_t)$ exists with probability $1$. On the
other hand, recurrence of $\pro X$ implies $\Prob(\liminf_{t\to\infty}\abs{X_t-x}=0)=1$. It then follows
that $\lim_{t\to\infty} \varphi(X_t)=\varphi(x)$, almost surely. By the conditional Fatou lemma applied to
$$
\ex^x\left[ \varphi(X_{t+s})|\cF_t\right]\leq \varphi(X_t)
$$
we get that $\varphi(x)\leq \varphi(X_t)$ almost surely. Since $\pro X$ has a positive transition density, this is
only possible when $\varphi$ attains its minimum at $x$. Since $x$ is arbitrary, we conclude that $\varphi$
is a constant function.
\end{proof}

Now we are ready to state our main result in this section.
\begin{theorem}[\textbf{Liouville theorem: transient case}]
\label{T4.9}
Let $\Psi \in \cB_0$ satisfy Assumptions~\ref{WLSC}(i) and ~\ref{WLSC}(ii) with $\underline{\theta} = 0 =
\bar\theta$, $\underline\mu > 0$ and $\bar\mu<\frac{d}{2}\wedge 1$. Furthermore, assume that $\pro X$ is
transient. Then there is no non-trivial, non-negative super-solution to
\begin{equation}\label{ET4.9A}
\Psidel\varphi\geq \varphi^p \quad \text{in}\;\; \Rd,
\end{equation}
for $p\in(1, \frac{d}{d-2{\underline\mu}})$.
\end{theorem}

To prove Theorem~\ref{T4.9} we need the following lemma. For $r>0$ define
$$
\sM(r) = \min_{\abs{x}\leq r} \varphi(x).
$$
\begin{lemma}\label{L4.4}
Suppose that the conditions in Theorem~\ref{T4.9} hold, and let $\varphi$ be a super-solution of
$$
\Psidel \varphi\geq 0 \quad \text{in}\;\; \Rd,
$$
Then we have
\begin{itemize}
\item[(i)]
For a constant $C_1$, independent of $x$, we have for $1<\abs{x}\leq r$
\begin{equation}\label{EL4.4A}
\Prob^x(\breve\uptau_1<\infty) \geq C_1 \frac{1}{r^d\Psi(r^{-2})},
\end{equation}
where $\Breve\uptau_1=\uptau_{\sB^c_1(0)}$ denotes the first entrance time to the ball $\sB_1(0)$.
\item[(ii)]
For a constant $C'_1$, independent of $u$, we have for $r>1$
\begin{equation}\label{EL4.4A1}
\sM(r)\geq C'_1\, \sM(1)\, \frac{1}{r^d\Psi(r^{-2})}\,.
\end{equation}
\item[(iii)]
For a constant $C_2$, independent of $u$, we have
\begin{equation}\label{EL4.4B}
\sM(r)\leq C_2 \, \sM(2r), \quad r\geq 2.
\end{equation}
\end{itemize}
\end{lemma}

\begin{proof}
In the proof below we will make use of the following two facts.
\begin{align}
\nu(x) &\asymp \frac{\Psi(\abs{x}^{-2})}{\abs{x}^d}, \quad  x\neq 0,\label{EL4.4C}
\\
\ex^x[\uptau_{\sB_r(x)}] & \asymp \frac{1}{\Psi(r^{-2})}, \quad r>0, \; x\in\Rd. \label{EL4.4D}
\end{align}
Relation \eqref{EL4.4C} follows from \cite[Cor.~23]{BGR14b}, and \eqref{EL4.4D} follows from
\cite[Rem.~4.8]{SL}. Note that concavity and monotonicity of $\Psi$ give
$$
\Psi(u)\leq \Psi(\delta u)\leq \delta \Psi(u), \quad  u>0, \; \delta\geq 1.
$$
Therefore, for every $\kappa > 0$ we have
\begin{equation}\label{EL4.4E}
\kappa\wedge 1\leq \inf_{u\in(0, \infty)} \frac{\Psi(\kappa u)}{\Psi(u)}\leq \sup_{u\in(0, \infty)}
\frac{\Psi(\kappa u)}{\Psi(u)}\leq 1\vee \kappa.
\end{equation}
Due to Assumption~\ref{WLSC}(i), for $\kappa\geq 1$ we have
\begin{equation}\label{Ex1}
\sup_{u\in(0, \infty)} \frac{\Psi(u)}{\Psi(\kappa u)}\leq \frac{1}{\underline{c}}\kappa^{-\underline\mu}.
\end{equation}
The proof of part (i) is based on \cite{GR14}. From \cite[Cor.~1.8]{Mimica} it is known that
\begin{equation}\label{Ex2}
G(x, y)=G(\abs{x-y})\asymp \frac{1}{\abs{x-y}^d \Psi(\abs{x-y}^{-2})}, \quad x\neq y,
\end{equation}
where $G$ is the Green function of $\pro X$ in $\Rd$. Denote by $G_{\cD}$ the Green function of the killed
process in $\cD$. Then by \eqref{Ex1}, \eqref{Ex2}, and a similar reasoning as in \cite[Prop.~4]{GR14} we obtain
\begin{equation}\label{Ex3}
G_{\sB_r(0)}(x, y)\geq\; \frac{1}{2} G(\abs{x-y}), \quad \text{whenever}\quad L\abs{x-y}\leq (r-\abs{x})\vee(r-\abs{y}),
\end{equation}
for a constant $L = L(d, \underline\mu, \underline{c}) >1$. Using \eqref{Ex3}, similarly to the argument in
\cite[Prop.~7]{GR14} we get that for large enough $R>r$
$$
\Prob^x(\Breve\uptau_1<\uptau_{\sB_R(0)})\geq \kappa_1 G(2 r) \mathrm{Cap}(\sB_1(0)),
$$
where $\mathrm{Cap}(\sB_1(0))$ denotes the capacity of $\sB_1(0)$, and $\kappa_1$ is a positive constant
independent of $R$. Thus by \cite[Prop.~3]{GR14} and \eqref{Ex2} we have
$$
\Prob^x(\breve\uptau_1<\uptau_{\sB_R(0)}) \geq
\kappa_2 \frac{\abs{\sB_1(0)} \Psi(\abs{\sB_1(0)}^{-\frac{2}{d}})}{r^d\Psi([2r]^{-2})},
$$
with a constant $\kappa_2$. Therefore, using \eqref{EL4.4E} and letting $R\to\infty$ in the above, we get
$$
\Prob^x(\breve\uptau_1<\infty) \geq C_1 \frac{1}{r^d\Psi(r^{-2})},\quad x\in \sB_r(0),
$$
with a constant $C_1$. This gives \eqref{EL4.4A}.

To show (ii), we consider the stopping time $\Breve\uptau_1$, the first hitting time of the ball $\sB_1(0)$.
Since $(\varphi(X_t))_{t\geq 0}$ is a non-negative super-martingale (see Theorem~\ref{T1.2} above), it is
easily seen that
$$
\varphi(x)\geq \ex^x[\varphi(X_{t\wedge\breve\uptau_{1}})], \quad t>0\,.
$$
Letting $t\to\infty$ and applying Fatou's lemma, we conclude that for $\abs{x}>1$
\begin{equation}\label{EL4.4A2}
\varphi(x) \geq \ex^x[\varphi(X_{\breve\uptau_{1}})\Ind_{\{\breve\uptau_1<\infty\}}]
\geq \sM(1) \Prob^x(\breve\uptau_1<\infty).
\end{equation}
Combining \eqref{EL4.4A2} with \eqref{EL4.4A}, we obtain \eqref{EL4.4A1}.

Finally, we prove part (iii). We split the proof into two cases. Suppose that $r<\abs{x}\leq \frac{3r}{2}$.
Consider the ball $\sB_{r/2}(x)$ of radius $\frac{r}{2}$ around $x$. Since $\varphi$ is a super-solution, we
have
$$
\varphi(x)\geq \ex^x[\varphi(X_{t\wedge\uptau_{\sB_{r/2}(x)}})], \quad  t>0.
$$
Hence by letting $t\to\infty$ and applying Fatou's lemma we obtain
$$
\varphi(x)\geq \ex^x[\varphi(X_{\uptau_{\sB_{r/2}(x)}})].
$$
By using the Ikeda-Watanabe formula \eqref{IW} we obtain
\begin{align*}
\varphi(x) &\geq \ex^x[\varphi(X_{\uptau_{\sB_{r/2}(x)}}) \Ind_{\sB_{r/2}(0)}(X_{\uptau_{\sB_{r/2}(x)}})]
\\
&\geq
\sM\left(\frac{r}{2}\right)\, \ex^x[\Ind_{\sB_{\frac{r}{2}}(0)}(X_{\uptau_{\sB_{r/2}(x)}})] \, = \,
\sM\left(\frac{r}{2}\right)\, P_{\sB_{r/2}(x)}(x, \sB_{r/2}(0))
\\
&\geq \kappa_1 \sM(r)\, \int_{\sB_{r/2}(x)}\int_{\sB_{r/2}(0)}\frac{\Psi(\abs{z-y}^{-2})}{\abs{z-y}^d}\, \D{z}
\, G_{\sB_{r/2}(x)}(x, \D{y})
\\
&\geq \kappa_1 \sM(r)\, \frac{\Psi(\abs{3r}^{-2})}{\abs{3r}^d}\abs{\sB_{r/2}(0)}
\int_{\sB_{r/2}(x)}\, G_{\sB_{r/2}(x)}(x, \D{y})
\\
&\geq \kappa_1 \sM(r) \frac{\abs{\sB_{r/2}(0)}}{\abs{3r}^d}\Psi(\abs{3r}^{-2}) \ex^{x}[\uptau_{\sB_{r/2}(x)}]
\\
&\geq \kappa_2 \sM(r) \frac{\abs{\sB_{r/2}(0)}}{\abs{3r}^d}\Psi(\abs{3r}^{-2}) \frac{1}{\Psi(\abs{\frac{r}{2}}^{-2})}
\\
&\geq \kappa_3 \sM(r),
\end{align*}
where in the third line we used monotonicity of $\sM$ and \eqref{EL4.4C}, and \eqref{EL4.4E} in the last line. For
the second case, assume $\frac{3r}{2}<\abs{x}\leq 2r$ and consider a ball of radius $\frac{r}{2}$ around $x$, and
a ball of radius $r$ around $0$. Then the proof of \eqref{EL4.4B} can be completed by repeating the same argument
as for the first case.
\end{proof}

Now we complete the proof of Theorem~\ref{T4.9}.
\begin{proof}[Proof of Theorem~\ref{T4.9}]
Let $r>0$  be large enough, and consider a point $\abs{x}\leq r$. Since $\varphi$ is a non-negative super-solution
of \eqref{ET4.9A}, we have
\begin{equation}\label{ET4.9B}
\varphi(x)\geq \ex^x[\varphi(X_{t\wedge\uptau_{\sB_r(x)}})] +
\ex^x\left[\int_0^{t\wedge\uptau_{\sB_r(x)}} \varphi^p(X_s)\, \D{s}\right], \quad t>0.
\end{equation}
Since $\varphi$ in non-negative, by letting $r\to\infty$ in \eqref{ET4.9B} we obtain
$$
\varphi(x)\geq \ex^x[\varphi(X_{t})] + \ex^x\left[\int_0^{t} \varphi^p(X_s)\, \D{s}\right], \quad t>0.
$$
Since $\pro X$ has a positive transition density, the above implies either $\varphi= 0$ or $\varphi>0$ in $\Rd$.
We show that the second possibility cannot occur. Suppose that $\varphi>0$. Then we have $\sM(r)>0$ for
$r>0$. Letting $t\to\infty$ in \eqref{ET4.9B}, we obtain
$$
\varphi(x)\geq  \ex^x\left[\int_0^{\uptau_{\sB_r(x)}} \varphi^p(X_s)\, \D{s}\right].
$$
This then implies
$$
\varphi(x)\geq (\sM(2r))^p \ex^x[\uptau_{\uptau_{\sB_r(x)}}]\geq \kappa_1 (\sM(2r))^p \frac{1}{\Psi(r^{-2})},
\quad \abs{x}\leq r,
$$
with a constant $\kappa_1$, where the last estimate follows from \eqref{EL4.4D}. Thus we have
$$
\kappa_1 (\sM(2r))^p \frac{1}{\Psi(r^{-2})} \leq \sM(r) \leq C_2 \sM(2r),
$$
which, in turn, gives
\begin{equation}\label{ET4.9C}
\sM(2r) \leq \kappa_2 \Psi(r^{-2})^{\frac{1}{p-1}}.
\end{equation}
Applying \eqref{EL4.4A1} to \eqref{ET4.9C} we get
\begin{equation}\label{ET4.9D}
r^{-d}\lesssim \Psi(r^{-2})^{\frac{p}{p-1}}, \quad \text{for all large}\; r.
\end{equation}
Since $\underline{\theta}=0$, in virtue of Assumption~\ref{WLSC}(i) we have that
\begin{equation}\label{ET5.40}
\underline{c}\,\Psi(r^{-2}) r^{2\underline\mu}\leq \Psi(1).
\end{equation}
Thus by using \eqref{ET4.9D} we get
$$
r^{-d} \lesssim r^{-\frac{2p\underline\mu}{p-1}}.
$$
However, this is not possible due to the fact that $d<\frac{2p\underline\mu}{p-1}$, i.e., $p <
\frac{d}{d-2\underline\mu}$. Thus necessarily $\sM(1)=0$, which contradicts the fact that $\varphi>0$,
and hence $\varphi=0$ in $\Rd$.
\end{proof}
We note that for the specific case $\Psi(u) = u$ and $\mu = \bar\mu = 1$, when the problem reduces to super-solutions of
a PDE featuring the Laplacian, and choosing $d \geq 3$, Theorem \ref{T4.9} covers the range $p \in (1, \frac{d}{d-2})$,
which can be compared with \cite{GS}, where solutions are analyzed. Also, recently in \cite[Th.~1.3]{FQ11} a similar range
of $p$ has been obtained for a more restricted class of operators (fractional Laplacian type operators), in particular,
corresponding to $\underline\mu =\bar\mu$.

We conclude with a further application of the technique developed in Theorem~\ref{T4.9} to Lane-Emden systems featuring
general non-local operators. This involves a system of coupled positive entire super-solutions. For some related results
involving the fractional Laplacian we refer to \cite{DKK, QX16}.

\begin{theorem}\label{T5.4}
Let $\Psi_1, \Psi_2 \in \cB_0$ be different Bernstein functions satisfying Assumptions~\ref{WLSC}(i) and \ref{WLSC}(ii)
with parameters $\underline{\mu}_i, \underline{c}_i > 0$, $\underline{\theta}_i=0$ and $\bar{\mu}_i, \bar{c}_i > 0$,
$\bar{\theta}_i=0$, $i=1,2$, respectively. Furthermore, assume that corresponding processes $\pro{X^1}$ and $\pro{X^2}$
satisfy the assumption of Theorem~\ref{T4.9}. Then there exists no positive super-solution to the system of equations
\begin{equation}\label{ET5.4A}
\begin{cases}
\Psi_1(-\Delta)\varphi_1 &\geq \; \varphi_2^p \quad \mbox{in}\;\; \Rd,
\\
\Psi_2(-\Delta)\varphi_2 &\geq \; \varphi_1^q \quad \mbox{in}\;\; \Rd,
\end{cases}
\end{equation}
whenever
\begin{equation}\label{ER5.2A}
\frac{2q\underline{\mu}_2+ 2pq\underline{\mu}_1}{pq-1} \,\vee\, \frac{2p\underline{\mu}_1+ 2pq\underline{\mu}_2}{pq-1} \;
> \; d.
\end{equation}
\end{theorem}

\begin{proof}
First note that either both $\varphi_1, \varphi_2$ are identically zero or positive in $\Rd$. Proceeding by
contradiction, we assume that both super-solutions are positive. Also, without loss of generality we assume that
\begin{equation}\label{ET5.4B}
\frac{2q\underline{\mu}_2+ 2pq\underline{\mu}_1}{pq-1} > d.
\end{equation}
Define $\sM_i(r)=\min_{\abs{x}\leq r}\varphi_i(x)$ for $i=1,2$. From the proof of Theorem~\ref{T4.9} we note that
\begin{equation}\label{ET5.4C}
\sM_1(r)\gtrsim \frac{1}{\Psi_1(r^{-2})} (\sM_2(2r))^p, \quad \sM_2(r)\gtrsim \frac{1}{\Psi_2(r^{-2})} (\sM_1(2r))^q\,,
\end{equation}
for $r>2$. Hence,
\begin{align*}
\sM_1(r) \gtrsim \frac{1}{\Psi_1(r^{-2})} (\sM_2(2r))^p
&\gtrsim \frac{1}{\Psi_1(r^{-2})} (\sM_2(r))^p
\\
&\gtrsim \frac{1}{\Psi_1(r^{-2})} \left[\frac{1}{\Psi_2(r^{-2})}\right]^q (\sM_1(2r))^{pq}
\\
&\gtrsim \frac{1}{\Psi_1(r^{-2})} \left[\frac{1}{\Psi_2(r^{-2})}\right]^q (\sM_1(r))^{pq},
\end{align*}
where in the first and third lines we used Lemma~\ref{L4.4}(iii), and \eqref{ET5.4C} in the second line.
This implies
\begin{equation}\label{ET5.4D}
\sM_1(r)\lesssim \left[\Psi_1(r^{-2}) \bigl(\Psi_2(r^{-2})\bigr)^q\right]^{\frac{1}{pq-1}}, \quad r>2.
\end{equation}
Thus using Lemma~\ref{L4.4}(ii) and \eqref{ET5.40} we obtain
$$
r^{-d}\lesssim r^{-\frac{2pq\underline{\mu}_1}{pq-1}} r^{-\frac{2q\underline{\mu}_2}{pq-1}}.
$$
This contradicts \eqref{ET5.4B}, hence there is no positive pair of super-solutions of \eqref{ET5.4A}.
\end{proof}

\begin{remark}
Condition \eqref{ER5.2A} is similar to the one obtained in \cite[Th.~2]{DKK} for fractional Laplace operators.
It should be noted that \cite{DKK} deals with weak solutions of fractional Laplacians, whereas our result deals
with super-solutions for a much larger class of operators.
\end{remark}

\begin{remark}
It is also easily seen that the conclusion of Theorem~\ref{T5.4} continues to hold if we replace \eqref{ET5.4A}
by a more general class of equations
\begin{equation*}
\begin{cases}
\Psi_1(-\Delta)\varphi_1 &\geq \; f(\varphi_2) \quad \mbox{in}\;\; \Rd,
\\
\Psi_2(-\Delta)\varphi_2 &\geq \; g(\varphi_1) \quad \mbox{in}\;\; \Rd,
\end{cases}
\end{equation*}
with $f(u)\gtrsim u^p$ and $g(u)\gtrsim u^q$.
\end{remark}

\bigskip
\subsection*{Acknowledgments}
\noindent
We are pleased to thank Xavier Cabr\'e for useful comments.
This research of AB was supported in part by an INSPIRE faculty fellowship and a DST-SERB grant EMR/2016/004810.
\bibliographystyle{abbrv}
\bibliography{BL}

\end{document}